\documentclass[reqno,11pt]{amsart}
\usepackage{amssymb,amsmath,amsthm}
\usepackage{amscd}
\usepackage{amsfonts}
\usepackage{amssymb}
\usepackage{latexsym}
\usepackage{color}
\usepackage{esint}
\usepackage{graphicx}
\usepackage{float}
\graphicspath{{Figures/}}

\usepackage{graphicx}
\usepackage[breaklinks=true,bookmarks=false]{hyperref}
\usepackage{cleveref}

\usepackage[makeroom]{cancel}

\newtheorem{theorem}{Theorem}

\newtheorem{lemma}{Lemma}
\newtheorem{proposition}[theorem]{Proposition}
\newtheorem{remark}{Remark}

\let\e=\varepsilon
\let\d=\delta

\let\p=\partial

\let\O=\Omega

\numberwithin{equation}{section}

\let\hide\iffalse
\let\unhide\fi

\newcommand{\R}{\mathbb{R}}
\renewcommand{\P}{\mathbf{P}}

\newcommand{\be}{\begin{equation}}
\newcommand{\bm}{\begin{multline}}
\newcommand{\ee}{\end{equation}}
\newcommand{\dd}{\mathrm{d}}

\newcommand{\Bes}{\begin{eqnarray*}}
\newcommand{\Ees}{\end{eqnarray*}}
\newcommand{\Be}{\begin{equation} }
\newcommand{\Ee}{\end{equation}}

\allowdisplaybreaks

\def\p{\partial}

\def\O{\Omega}
\def\R{\mathbb{R}}
\def\d{\mathrm{d}}
\def\B{\begin{equation}}
\def\E{\end{equation}}
\def\BN{\begin{eqnarray*}}
\def\EN{\end{eqnarray*}}

\usepackage{color}
\newcommand{\Red}{\textcolor{red}}
\newcommand{\Blue}{\textcolor{blue}}

\begin{document}
\title[Diffusive limit of Boltzmann around Rayleigh profile]{Diffusive limit of the Boltzmann equation around Rayleigh profile in the half space}

\author[H.-X. Chen]{Hongxu Chen}
\address[HXC]{School of Mathematical Sciences, Shenzhen University, Shenzhen, Guangdong 518060, China}
\email{hongxuchen.math@gmail.com}

\author[R.-J. Duan]{Renjun Duan}
\address[RJD]{Department of Mathematics, The Chinese University of Hong Kong, Shatin, N.T., Hong Kong.}
\email{rjduan@math.cuhk.edu.hk}
 
\date{\today}
\subjclass[2020]{35Q20, 76P05}


\keywords{Boltzmann equation, half-space, moving boundary, Rayleigh solution, diffusive limit}
 
\begin{abstract}
This paper concerns the diffusive limit of the time evolutionary Boltzmann equation in the half space $\mathbb{T}^2\times\mathbb{R}^+$ for a small Knudsen number $\varepsilon>0$. For boundary conditions in the normal direction, it involves diffuse reflection moving with a tangent velocity proportional to $\varepsilon$ on the wall, whereas the far field is described by a global Maxwellian with zero bulk velocity. The incompressible Navier-Stokes equations, as the corresponding formal fluid dynamic limit, admit a specific time-dependent shearing solution known as the Rayleigh profile, which accounts for the effect of the tangentially moving boundary on the flow at rest in the far field. Using the Hilbert expansion method, for well-prepared initial data we construct the Boltzmann solution around the Rayleigh profile without initial singularity over any finite time interval.
\end{abstract}

\maketitle

\vspace{-6mm}
\begin{center}
\small\it Dedicated to Professor Shih-Hsien Yu on the occasion of his 60th birthday
\end{center}

\thispagestyle{empty}
\bigskip
\section{Introduction}

\subsection{Problem setting}

Consider a viscous fluid initially at rest above an infinite flat plate. At time $t=0$, the plate is impulsively put in motion with a constant velocity parallel to its surface. This sudden acceleration generates a shear flow that propagates into the fluid, characterized by a thin layer near the plate where viscous forces dominate. This example of viscous fluid motion is classically known as the Rayleigh problem \cite{rayleigh1966scientific}.

Although the problem was investigated by Rayleigh using the viscous incompressible fluid, we are interested in the rarefied gas flow governed by the Boltzmann equation around the incompressible fluid for this problem in kinetic theory. Let the flow be confined in the half-space $\O:=\mathbb{T}^2\times \mathbb{R}^+$, where the tangent direction is spatially periodic, then we consider the nonlinear Boltzmann equation under the diffusive scaling with a small parameter $\e>0$,
\begin{align}
  \displaystyle \e \p_t F +     v\cdot \nabla_x F = \frac{1}{\e}Q(F,F).
\label{eqn_F}  
\end{align}
Here, the unknown $F=F(t,x,v)\geq 0$ stands for the velocity distribution function of gas particles with velocity $v=(v_1,v_2,v_3)\in \R^3$ at time $t> 0$ and position $x=(x_1,x_2,x_3)\in \O$, supplemented with initial condition,
\begin{align}\label{ic}
 F(0,x,v) = F_0(x,v).
\end{align}
The boundary condition is to be specified in \eqref{bc} later on. The small parameter $\e>0$ denotes the Knudsen number which is proportional to the mean free path. Here and to the end, we would omit the explicit dependence of $F$ on $\e$ and thus write $F$ instead of $F^\e$ for the sake of simplicity of notation. The Boltzmann collision term is a bilinear integral operator acting only on the velocity variable, and for the hard sphere model, it reads as 
\begin{equation*}
Q(F,G)(v)=\int_{\mathbb{R}^3}\int_{\mathbb{S}^2}|(v-u)\cdot \omega|[F(u')G(v')-F(u)G(v)]\,\d\omega \dd u,
\end{equation*}
where the velocity pairs $(v,u)$ and $(v',u')$ satisfy 
\begin{equation*}
  v'=v-[(v-u)\cdot\omega]\omega,\quad u'=u+[(v-u)\cdot\omega]\omega,\quad \omega\in\mathbb{S}^2,
\end{equation*}
that's the $\omega$-representation in terms of the conservation of momentum and energy for elastic collisions between molecules:
\begin{equation*}
  v+u=v'+u',\quad |v|^2+|u|^2=|v'|^2+|u'|^2.
\end{equation*}

For $t>0$, at the boundary $x_3=0$, the moving plate is modeled by the diffuse reflection boundary condition
\begin{align}
    &  F(t,x_1,x_2,0,v)|_{v_3>0} = M_w \int_{v_3'<0} F(t,x_1,x_2,0,v') |v_3'|\, \dd v'.
    \label{bc}
\end{align}
Here $M_w$ is the wall Maxwellian given by
\begin{align}
    &   M_w = \frac{1}{2\pi} \exp \Big(-\frac{|v-u_w|^2}{2}\Big), \label{wall_maxwell}
\end{align}
where $u_w$ is the plate's velocity, which is parallel to the surface. Without loss of generality, we assume that the plate is moving in the $x_1$ direction and the moving velocity is given by 
\begin{align}
    &   u_w = (\e u_b , 0, 0)=\e (u_b,0,0), \label{def.bv}
\end{align}
for $0<u_b\ll 1$. Note that $u_w$ is proportional to $\varepsilon$.

In the far field, the gas is not affected by the boundary tangent motion and hence remains in a rest uniform equilibrium state, which corresponds to the global Maxwellian
\begin{align}\label{def.mu}
    \mu(v) = \frac{1}{(2\pi)^{3/2}} e^{-\frac{|v|^2}{2}}.
\end{align}
This leads to the far-field condition that for $t>0$,
\begin{align}
    F(t,x,v)\to \mu(v) \text{ as } x_3\to \infty. 
    \label{far_field}
\end{align}

In this article, for the initial boundary value problem as specified above, we are interested in the hydrodynamic diffusive limit behavior of the Boltzmann equation when the physical boundary is moving on its own plane but the far field is at rest.

\subsection{Hilbert expansion}
We express $F$ using the classical Hilbert expansion,
\begin{align}
    & F = \mu +  \sqrt{\mu} (\e f_1 + \sum_{i=2}^4 \e^i(f_i + \mathfrak{f}_i)  + \e^5 f_5) + \e^{5/2+c_0} \sqrt{\mu} R, \label{expansion}
\end{align}
where $f_i$ are the corrections up to the fifth order and $R$ denotes the remainder. Here $\mathfrak{f}_i=\mathfrak{f}_i(\frac{x_3}{\e})$ is the boundary layer equation that vanishes in the interior. The extra constant $c_0$ will be chosen as $c_0=\frac{1}{4}$ to control the nonlinear operator, we refer to detailed discussion in Remark \ref{rmk:c0}. To the end, we also omit the dependence of $R$ on $\e$ for brevity.

\textbf{Interior expansion}.

We first derive the interior equation $f_i$. Plugging \eqref{expansion} into \eqref{eqn_F} and comparing the order, we have
\begin{align}
    &    \text{ Order of $\e^0$: } \mathcal{L}f_1 = 0, \label{order_0}\\
    &    \text{ Order of $\e^1$: } v\cdot \nabla_x f_1 + \mathcal{L} f_2 = \Gamma(f_1,f_1),  \label{order_1}\\
    & \vdots \notag\\
    &    \text{ Order of $\e^{i}$: } \p_t f_{i-1} + v\cdot \nabla_x f_i + \mathcal{L}f_{i+1}   =\sum_{j+k=i+1,\ 1\leq j,k \leq i}\Gamma(f_j,f_k)  \label{order_i} , \\
    & \vdots \notag\\
    & \text{ Order of $\e^5$: }  \mathbf{P}(\p_t f_4 + v\cdot \nabla_x f_5) = 0. \label{order_5}
\end{align}
Here, $\mathcal{L},\Gamma$ denote the linearized and nonlinear Boltzmann operator:
\begin{align}
\begin{cases}
        &\displaystyle \mathcal{L}f := -\frac{Q(\mu,\sqrt{\mu}f) + Q(\sqrt{\mu}f,\mu)}{\sqrt{\mu}}, \\
    &\displaystyle \Gamma(f,g) = \Gamma(g,f) := \frac{Q(\sqrt{\mu}f,\sqrt{\mu}g)+Q(\sqrt{\mu}g,\sqrt{\mu}f)}{2\sqrt{\mu}}.
\end{cases} \label{operator}
\end{align}
Note that $\mathcal{L}$ is a nonnegative definite self-adjoint operator on $L^2_v$ with the kernel $\ker (\mathcal{L})$ spanned by the orthogonal set $\{1,v,|v|^2-3\}\sqrt{\mu}$. Recall the classical Grad's decomposition $\mathcal{L}=\nu-\mathcal{K}$ with the multiplication operator $\nu=\nu(v)\sim \sqrt{1+|v|^2}$ and the integral operator $\mathcal{K}$. For later use, we denote $\mathbf{P}$ to be the macroscopic projection onto $\ker (\mathcal{L})$ in a standard way:
\begin{align*}
    & \mathbf{P}f := \Big(a + \mathbf{b}\cdot v + c \frac{|v|^2-3}{2} \Big) \sqrt{\mu(v)} , \notag \\
    & a = \int_{\mathbb{R}^3} \sqrt{\mu} f \dd v, \ \mathbf{b} = \int_{\mathbb{R}^3} v\sqrt{\mu} f\dd v , \ c = \int_{\mathbb{R}^3} \frac{|v|^2-3}{2}\sqrt{\mu} f \dd v. \notag
\end{align*}

In what follows we construct $f_i$. First of all, \eqref{order_0} implies that 
\begin{align*}
    &f_1 = \Big(\rho + u\cdot v + \theta \frac{|v|^2-3}{2}\Big) \sqrt{\mu}.
\end{align*}
From a standard computation, \eqref{order_1} leads to
\begin{align}\label{pf10}
    & \mathbf{P}(v\cdot \nabla_x f_1) = 0, 
\end{align}
or equivalently one has the following incompressibility condition and Boussinesq relation:
\begin{align*}
    & \nabla_x \cdot u = 0, \ \nabla_x(\rho+\theta) = 0.
\end{align*}
Since the wall temperature is identical to $1$, in this paper we do not consider the contribution of the energy $\theta$ and set $\rho = \theta = 0$. This leads to the expression of $f_1$:
\begin{align}
    &   f_1 = u\cdot v\sqrt{\mu}. \label{f1}
\end{align}

From \eqref{order_1} together with \eqref{pf10}, we further write
\begin{align}
\begin{cases}
 &\displaystyle (\mathbf{I}-\mathbf{P})f_2 = -\mathcal{L}^{-1}(v\cdot \nabla_x f_1) + \mathcal{L}^{-1}(\Gamma(f_1,f_1)),     \\
&\displaystyle  \mathbf{P}f_2 = \Big(\rho_2 + u_2\cdot v +  \theta_2 \frac{|v|^2-3}{2}\Big)\sqrt{\mu}.
\end{cases} 
\label{f2_derivation}
\end{align}
With this, $\mathbf{P}(\eqref{order_i})=0$ with $i=2$ leads to 
\begin{align*}
    & \p_t \rho = \nabla_x \cdot u_2 = 0
\end{align*}
and the incompressible Navier-Stokes equations (see \cite{BGL:91} for detail):
\begin{align}
\begin{cases}
        &\displaystyle  \p_t u + u\cdot \nabla_x u + \nabla_x p = \kappa \Delta u, \ \nabla\cdot u = 0,  \\   
    &\displaystyle   p = \rho_2 + \theta_2 - \frac{1}{3}|u|^2, 
\end{cases}\label{ins}
\end{align}
Here $\kappa>0$ is the viscosity constant given by
\begin{align}
\begin{cases}
        &  \langle A_{ij},\mathcal{L}^{-1}A_{kl}\rangle = \kappa \left( \delta_{ik}\delta_{jl} + \delta_{il}\delta_{jk}-\frac{2}{3}\delta_{ij}\delta_{kl} \right),  \\
    & A_{ij}:= v_iv_j -\delta_{ij}\frac{|v|^2}{3}. 
\end{cases}\label{Aij}
\end{align}

We are interested in the Rayleigh profile of \eqref{ins}. For this purpose, we specify the boundary condition and initial condition for $f_1$ here, the conditions for $f_i,i>1$ will be specified later in the boundary layer expansion. We impose a tangent moving speed on the boundary for $u$: 
\begin{align}
u(x_1,x_2,0) = (u_b,0,0). \notag 
\end{align} 
Moreover, instead of the instant moving of the boundary in the classical Rayleigh problem, we impose the well-prepared smooth initial condition as 
$$
u(t,x)|_{t=0} = \Big(\frac{2u_b}{\sqrt{\pi}} \int^\infty_{\frac{x_3}{\sqrt{4\kappa \delta}}} e^{-r^2}\dd r,0,0 \Big)
$$ 
to avoid singularity at $t=0$, where $\delta>0$ is a fixed constant throughout the paper. Therefore, the initial boundary value problem on $u$ is given by
\begin{align}
\begin{cases}
      &\displaystyle  \p_t u + u\cdot \nabla_x u + \nabla_x p = \kappa \Delta u, \ \ \nabla_x \cdot u = 0,\\
      & \displaystyle \nabla_x p = \nabla_x (\rho_2+\theta_2- \frac{1}{3}|u|^2),  \\[2mm]
    & \displaystyle  \ u(t,x)|_{t=0} = \Big(\frac{2u_b}{\sqrt{\pi}} \int^\infty_{\frac{x_3}{\sqrt{4\kappa \delta}}} e^{-r^2}  \dd r  ,0,0 \Big), \\     & \displaystyle u(t,x)|_{x_3=0} = (u_b,0,0) .  
\end{cases} \label{INS_full}
\end{align}
To fix a specific shearing solution of \eqref{INS_full}, we first take the velocity field as
\begin{align}
    & u = (u_p,0,0), \ \ u_p:= \frac{2u_b}{\sqrt{\pi}} \int^{\infty}_{\frac{x_3}{\sqrt{4\kappa (t+\delta)}}} e^{-r^2} \dd r, \label{u_p}
\end{align}
such that $u_p=u_p(t,x_3)$ satisfies the heat equation
$$
\p_tu_p=\kappa \p_{x_3x_3}u_p.
$$
This provides a complete description of $f_1$ in \eqref{f1}. To solve \eqref{order_i} for $O(\e^2)$, we set
\begin{align*}
    &    (\mathbf{I}-\mathbf{P})f_3 = -\mathcal{L}^{-1}(\p_t f_{1} + v\cdot \nabla_x f_{2}) + \mathcal{L}^{-1}(2\Gamma(f_1,f_2)).
\end{align*}

Then we derive $f_i$ using \eqref{order_i} in an inductive way. Assume \eqref{order_i} is solved for $O(\e^{i}),i\geq 2$, we solve $f_{i+1}$ using \eqref{order_i} as
\begin{align}
   & \begin{cases}
        &   (\mathbf{I}-\mathbf{P})f_{i+1} = \mathcal{L}^{-1}(-\p_t f_{i-1} - v\cdot \nabla_x f_{i} +\sum_{j+k=i+1}\Gamma(f_j,f_k)), \\
        & \mathbf{P}f_{i+1} = \Big(\rho_{i+1} + u_{i+1}\cdot v + \theta_{i+1} \frac{|v|^2-3}{2}\Big). \sqrt{\mu}
    \end{cases}   \label{f_i}
\end{align}
Plugging this into $\mathbf{P}(\eqref{order_i})=0$ at order $O(\e^{i+1})$, we derive the linearized Navier Stokes Fourier system(see \cite{guo2006boltzmann} for derivation)
\begin{align}
    &  \begin{cases}
        & \p_t P_0u_{i} + u\cdot \nabla_x P_0u_{i} + P_0u_{i} \cdot \nabla_x u + \nabla_x p_{i}  = \kappa P_0\Delta u_i + S_u^{i}, \  \\
        &\nabla_x \cdot (I-P_0)u_{i} = \p_t \rho_{i-1}, , \ \  p_{i} = \rho_{i+1}+\theta_{i+1} - \frac{2}{3}u_{i}\cdot u, \\
        & \p_t \theta_{i} + u\cdot \nabla_x \theta_{i} = \eta \Delta \theta_{i} + S_\theta^i, \\
        & \p_t \rho_{i} = \nabla_x \cdot u_{i+1}.
    \end{cases} \label{u_theta_i_eqn}
\end{align}
Here $P_0 u_i$ is the divergence free projection of $u_i$. The source term are given by \eqref{S_u} and \eqref{S_theta}.

Once $u_i,\theta_i$ are solved in \eqref{u_theta_i_eqn} with specified initial condition and boundary condition, with $\rho_i+\theta_i = \frac{2-\delta_{i-1}^1}{3}u_{i-1}\cdot u$ from \eqref{INS_full} and \eqref{u_theta_i_eqn}, we can determine $\rho_i$ as
\begin{align}
\rho_i = p_{i-1}+\frac{2-\delta_{i-1}^1}{3}u_{i-1}\cdot u - \theta_i, 
\label{r2t2}
\end{align}

\eqref{order_5} is solved by the same way as \eqref{f_i} with $i=4$. The macroscopic components of $f_5$ can be chosen with some freedoms. To guarantee the condition $\p_t \rho_4 = \nabla_x \cdot u_5$ and $\rho_5 + \theta_5 = p_4 + \frac{2}{3}u_4\cdot u$, we set 
\begin{align}
    &  \theta_5 = 0 , \ \ \rho_5 = p_4 + \frac{2}{3}u_4\cdot u, \ \  u_5 = \nabla_x  \Delta^{-1}(\p_t \rho_4). \label{f5}
\end{align}

\textbf{Boundary layer expansion.}

Next, we focus on the boundary condition. We expand the wall Maxwellian as
\begin{align}
    &   \frac{M_w}{\sqrt{\mu}} = \frac{M_w - c_\mu \mu + c_\mu \mu}{\sqrt{\mu}}   = \sum_{i=0}^4   c_\mu \e^i m_i \sqrt{\mu} + O(\e^5)\langle v\rangle^5\sqrt{\mu(v)}. \label{wall_expand}
\end{align}
Here $m_i$ is some $i$-order polynomial in $v_1$:
\begin{align*}
    &  m_0 = 1, \ m_1 = u_b v_1 , \ m_2 = \frac{|u_b v_1|^2 - |u_b|^2}{2}, \cdots.
\end{align*}

Plugging \eqref{expansion} into the boundary condition \eqref{bc}, we have
\begin{align*}
    & \e f_1 + \sum_{i=2}^4 \e^i (f_i +  \mathfrak{f}_i)+ \e^5 f_5 +  \e^{5/2+c_0} R \\
    & = \frac{M_w}{\sqrt{\mu}} \int_{v_3' > 0}\sqrt{\mu} \big[\e f_1 + \sum_{i=2}^4 \e^i (f_i +  \mathfrak{f}_i)+ \e^5 f_5 +  \e^{5/2+c_0} R \big] v_3' \dd v' +   \frac{\frac{1}{c_\mu}M_w - \mu}{\sqrt{\mu}}.
\end{align*}
Denote
\begin{align*}
    & P_\gamma f = c_\mu\sqrt{\mu}\int_{v_3'<0} \sqrt{\mu} f |v_3'| \dd v'.
\end{align*}

Using the expansion \eqref{wall_expand} and comparing the order at $O(\e)$, we obtain the boundary condition for $f_1$:
\begin{align*}
    &   f_1  =  P_\gamma f_1 + m_1 \sqrt{\mu}.
\end{align*}
This leads to the boundary condition as in \eqref{INS_full}.
\begin{align}
    &  u |_{x_3=0} = (u_b,0,0) \notag.
\end{align}

Next, the boundary condition for the order $O(\e^i),2\leq i\leq 4$ reads
\begin{align}
    &   [f_i + \mathfrak{f}_i]|_{\gamma_-} = P_\gamma(f_i + \mathfrak{f}_i) + m_i \sqrt{\mu}. \label{bc_2}
\end{align}
The boundary layer $\mathfrak{f}_i$ needs to correct the mismatch from $(\mathbf{I}-\mathbf{P})f_i$ and $m_i \sqrt{\mu}$. We denote
\begin{align}
    & \mathfrak{h}_i := P_\gamma((\mathbf{I}-\mathbf{P})f_i) - (\mathbf{I}-\mathbf{P})f_i + m_i \sqrt{\mu}. \label{hi}
\end{align}
We note that $\mathfrak{h}_i$ is determined by $u_b$ and $(\mathbf{I}-\mathbf{P})f_i$, then by \eqref{f_i}, it is determined by $u_b$ and $f_{i-1},f_{i-2},\cdots.. f_1$.

WLOG, we set $\mathfrak{f}_i = 0$ for $i\leq 1$. Denote $\xi = \frac{x_3}{\e}$, then $\mathfrak{f}_i$ should satisfy
\begin{align}
    &  \begin{cases}
        v_3 \frac{\p \mathfrak{f}_i}{\p \xi} + \mathcal{L}\mathfrak{f}_i &= -\p_t \mathfrak{f}_{i-2} - [v_1\p_{x_1} + v_2\p_{x_2}] \mathfrak{f}_{i-1}  \\
        &+ \sum_{j+k=i, \ 1\leq j,k\leq i-1}[\Gamma(f_j+\mathfrak{f}_j,\mathfrak{f}_k) + \Gamma(\mathfrak{f}_j,f_k)], \\
        \mathfrak{f}_i(0,v)& = \mathfrak{h}_i - \mathfrak{M}^i_\infty(v), \  \ \lim_{\xi\to\infty}\mathfrak{f}_i(\xi,v) = 0 
    \end{cases} \label{f_i_bl}
\end{align}
Here 
\begin{equation}
\mathfrak{M}^i_\infty(v) : = \big[ \rho^i_\infty + u^i_\infty  \cdot v + \theta^i_\infty \frac{|v|^2-3}{2} \big] \sqrt{\mu(v)}\in \ker \mathcal{L}  
\end{equation}
with $u^i_{\infty,3} = 0$, and $\mathfrak{M}^i_\infty(v)$ can be determined by solving the Milne's problem as \cite{bardos1986milne}.

To cancel the $v_i, |v|^2$ component in \eqref{bc_2}, we need to impose the following boundary condition to $u_i$ and $\theta_i$:
\begin{align}
    &   u_i|_{x_3=0} = u^i_\infty, \ \ \   \theta_i|_{x_3= 0} = \theta^i_\infty. \label{u_theta_i_bc}
\end{align}
The boundary condition for $\rho_i$ enjoys extra freedom
\begin{align}
    & \rho_i |_{x_3=0} = \rho_\infty^i + \varrho_i. \label{rho_i_bc}
\end{align}
Here $\varrho_i$ is chosen to satisfy the Boussinesq relation \eqref{r2t2}. Such a choice of boundary condition guarantees \eqref{bc_2}, see \cite{wu2022hydrodynamic}.

In summary, by \eqref{hi} and \eqref{f_i_bl}, we conclude that $\mathfrak{f}_i$ and the boundary condition \eqref{u_theta_i_bc}, \eqref{rho_i_bc} are determined by $f_1,\cdots,f_{i-1}$ and $\mathfrak{f}_{i-1},\cdots,\mathfrak{f}_2$. Inductively, $\mathfrak{f}_i$ and \eqref{u_theta_i_bc}, \eqref{rho_i_bc} are determined by $f_1,\cdots,f_{i-1}$. 

Now the boundary condition of the linearized NSF system \eqref{u_theta_i_eqn} are determined, since the source term \eqref{S_theta}, \eqref{S_u} are again determined by $f_1,\cdots,f_{i-1}$, we conclude that $f_i$ are determined by $f_1,\cdots,f_{i-1}$ and the initial condition $f_{i,0}$. 

Finally, to avoid singularity, we impose the following compatibility condition:
\begin{align}
    & (\rho_i, u_{i}, \theta_i)|_{t=0,x_3=0} = (\rho^i_\infty, u^i_\infty,\theta^i_\infty)|_{t=0}, \ \ 2\leq i\leq 4. \label{compatibility}
\end{align}

\textbf{Remainder formulation.}

Now $f_i,\mathfrak{f}_i$ have been completely specified in the Hilbert expansion \eqref{expansion}. To justify the hydrodynamic limit, the main goal is to obtain proper estimates of the remainder term. We focus on the formulation of the remainder term $R$ in \eqref{expansion}. From \eqref{far_field}, we have the following condition at the far field:
\begin{align*}
    & R(x) \to 0 \text{ as }|x|\to \infty.
\end{align*}

Plugging the expansion \eqref{expansion} into the original equation \eqref{eqn_F}, we obtain the following equation for $R$:
\begin{align}
    & \p_t R + \frac{1}{\e}v\cdot \nabla_x R + \frac{1}{\e^2} \mathcal{L}R = h + \frac{1}{\e}\tilde{\mathcal{L}}R + \e^{1/2+c_0} \Gamma(R,R). \label{remainder_eqn}
\end{align}
In \eqref{remainder_eqn}, we have denoted the linear term
\begin{align}
    & \tilde{\mathcal{L}}R := 2\Gamma(f_1 +  \sum_{i=2}^4 \e^{i-1}(f_i + \mathfrak{f}_i) + \e^4 f_5 , R) .  \label{L_tilde}
\end{align}

$h$ denotes the inhomogeneous source term. By \eqref{order_0}, \eqref{order_1}, \eqref{order_i} and \eqref{order_5}, we have the following expression for $h$:
\begin{align}
   h &   = \e^{1/2-c_0}\Big(- \p_t\mathfrak{f}_3 - \e \p_t \mathfrak{f}_4 - (v_1\p_{x_1}+v_2\p_{x_2})\mathfrak{f}_4  \Big) \notag \\
   & + \e^{1/2-c_0} \sum_{j+k\geq 5, \ 1\leq j,k} \e^{j+k-5}[\Gamma(f_j+\mathfrak{f}_j,\mathfrak{f}_k) + \Gamma(\mathfrak{f}_j,f_k) ]   \notag\\
   & + \e^{1/2-c_0}\sum_{j+k\geq  6, 1\leq j,k}\e^{j+k-5}\Gamma(f_j,f_k) - \e^{3/2}(\p_t f_4 + v\cdot \nabla_x f_5+\e\p_t f_5)  .\label{h_def}
\end{align}

\begin{remark}\label{rmk:c0}
The purpose of introducing an extra $c_0$ in \eqref{expansion} is to provide sufficient control for the nonlinear term in \eqref{remainder_eqn}, while maintaining the positive exponent in \eqref{h_def}.
\end{remark}

For the boundary condition, from \eqref{bc_2} and \eqref{wall_expand}, we obtain:
\begin{align}
    & R(t,x_1,x_2,0,v)|_{v_3>0} \notag \\
    &= P_\gamma R + \frac{M_w-c_\mu \mu}{\sqrt{\mu}} \int_{v_3'<0} R(t,x_1,x_2,0,v') \sqrt{\mu(v')} |v_3'| \dd v' + r ,    \label{R_bc}
\end{align}
where
\begin{align}
    &  r := \e^{5/2-c_0}\frac{M_w}{\sqrt{\mu}}\int_{v_3'<0} f_5\sqrt{\mu(v')} |v_3'| \dd v' - \e^{5/2-c_0}f_5 + O(\e^{5/2-c_0})\langle v\rangle^5\sqrt{\mu(v)}. \label{r_bc}
\end{align}

\subsection{Motivation}

The Rayleigh problem was initially studied in the context of incompressible fluid in \cite{rayleigh1966scientific}. Recently, Maekawa studied in \cite{maekawa2021gevrey} the local in time stability of the incompressible Navier-Stokes equations around the Rayleigh profile with initial singularity $\delta=0$ and vanishing viscosity $\kappa\to 0$. In his work, two-dimensional half-space domain $\mathbb{T}\times \R^+ $ is considered and the Gevrey regularity is imposed on the tangential direction. The current work is partially inspired by \cite{maekawa2021gevrey}; see Remark \ref{rmk2} and Remark \ref{rmk:higher_order_expansion} for more discussions. In particular, we could expect an analogous result for the kinetic Boltzmann equation in a similar framework. 

For rarefied gas, the Rayleigh problem can be studied using the kinetic equation; we refer the earlier investigation to Sone \cite{sone1964kinetic}. Recently, Kuo constructed a local-in-time solution to the Boltzmann equation and estimated the flow velocity in the Rayleigh profile in \cite{kuo2010initial}. In subsequent work \cite{kuo2017asymptotic}, Kuo analyzed the long-time behavior of the linearized Boltzmann equation around the Rayleigh profile in the hydrodynamic limit through asymptotic analysis. Building on these studies, in this paper we rigorously justify the hydrodynamic limit of the full Boltzmann equation near the Rayleigh profile by employing the recent $L^2-L^\infty$ argument with high order Hibert expansion. The constraints of our results are discussed in Remarks \ref{rmk1} and \ref{rmk2}. For instance, we have removed the initial singularity for which more subtle estimates are needed. 

The $L^2-L^\infty$ framework, introduced by Guo in \cite{G}, provides an effective approach to handling boundary value problems for the Boltzmann equation. This fundamental work has led to substantial developments in kinetic theory, including the stationary problem \cite{EGKM} and the regularity problem \cite{GKTT}. Esposito-Guo-Kim-Marra applied this method in \cite{EGKM2} to study the Navier-Stokes-Fourier limit in bounded domains, where an $L^6$ estimate was introduced to address challenges arising in the small Knudsen number regime. Further applications include the incompressible Euler limit, as explored in \cite{jang2021incompressible} and \cite{CKJ}.

The Rayleigh problem is a classic example of shear flow and can be viewed as a limiting case of Couette flow. In Couette flow, a gas is confined between two parallel plates at $x_3=\pm 1$ that are moving in opposite directions. When the upper plate is moved to infinity $(x_3\to \infty)$, the Couette flow reduces to the Rayleigh problem. In fluid dynamics, the Couette flow problem is proven to exhibit a stability mechanism, see the survey \cite{bedrossian2019stability} and references therein. In kinetic theory, Duan et al. \cite{duan2023boltzmann} \cite{duan20243d} have constructed the stationary solution and global dynamical stability of the Boltzmann equation in a Couette flow profile when the shear strength is small enough. Also see the infinite layer problem in \cite{chen2025global}. In contrast, the heat equation solution \eqref{u_p} is self-similar, implying that no stationary solution exists for the Rayleigh problem. A more detailed informal discussion of stationary solutions from the kinetic equation perspective can be found in Remark \ref{rmk1} and Appendix \ref{sec:stationary}.

\subsection{Main result}

Denote an exponential velocity weight,
\begin{align}
    & w(v) := e^{\beta |v|^2}, \ 0<\beta \leq \frac{1}{8}.  \label{w_weight}
\end{align}

We state the main result of this paper. We refer the norms notation to Section \ref{sec:notation}. 

\begin{theorem}\label{thm:well-posedness}
Let $c_0=\frac{1}{4}$. Let the non-singular Rayleigh profile $u=(u_p,0,0)$ be defined in \eqref{u_p} and $f_1,f_i$ be given in \eqref{f1} and \eqref{f_i}, respectively. Suppose the initial condition of \eqref{eqn_F} is given in the well-prepared form
\begin{align*}
    & F_0(x,v) = \mu + \sqrt{\mu}\Big(\e f_{1,0} + \sum_{i=2}^4 \e^i(f_{i,0}+\mathfrak{f}_{i,0}) +\e^5 f_{5,0} + \e^{5/2+c_0} R_0\Big)\geq 0, 
\end{align*}
where $f_{i,0},\mathfrak{f}_{i,0}$ are the initial conditions of $f_i$ and $\mathfrak{f}_i$, respectively. We assume $(\rho_{i,0},u_{i,0},\theta_{i,0})$ satisfies the compatibility condition \eqref{compatibility} and the boundedness
\begin{align}
    &  \sum_{i=2}^5 \sum_{j+k\leq 8-i} \Vert \p_t^j \nabla_{x}^k(\rho_{i,0},u_{i,0},\theta_{i,0})\Vert_{L^2_{x,v}}  < \infty. \label{f_i_intial}
\end{align}
Then for any given time $T>0$ and given constant $C_0>0$, there exists $\e_0=\e_0(T,C_0) \ll 1$ such that if $\e < \e_0$, if $u_b>0$ satisfies
$u_bT^{1/2} \ll 1$, and if $R_0$ satisfies
\begin{align}
    & \Vert R_0\Vert_{L^2_{x,v}} + \Vert \e^{3/2}wR_0\Vert_{L^\infty_{x,v}}     < C_0 ,  \label{R_initial}
\end{align}
then there exists a unique solution
\begin{equation*}
F=\mu +  \sqrt{\mu} (\e f_1 + \sum_{i=2}^4 \e^i(f_i + \mathfrak{f}_i) + \e^5 f_5 + \e^{5/2+c_0} R)    \geq 0
\end{equation*}
on $[0,T]$ to the initial-boundary value problem \eqref{eqn_F}, \eqref{ic}, \eqref{bc}. Moreover, it holds
\begin{align}
    &\Vert R\Vert_{L^\infty_t L^2_{x,v}} + \Vert \e^{3/2}wR\Vert_{L^\infty_{t,x,v}} + \e^{-1/2}|(I-P_\gamma)R|_{L^2_{t,\gamma_+}} \notag\\
    &+ \e^{-1}\Vert (\mathbf{I}-\mathbf{P})R\Vert_{L^2_{t,x,\nu}} < C_1C_0, \label{R_est}
\end{align}
for any $0\leq t\leq T$, where $C_1$ is a constant independent of $T$. Additionally, we have the following $L^2_{x,v}$ and $L^\infty_{x,v}$-convergence in $\e$:
\begin{align}
    &  \sup_{0\leq t\leq T}\Big\Vert \frac{F(t)-\mu}{\e\sqrt{\mu}}- f_1(t) \Big\Vert_{L^2_{x,v}} \lesssim \e, \notag\\
    & \sup_{0\leq t\leq T}\Big\Vert \frac{F(t)-\mu}{\e\sqrt{\mu}}- f_1(t) \Big\Vert_{L^\infty_{x,v}}  \lesssim \e^{\frac{1}{4}}.  \label{convergence}
\end{align}

\end{theorem}

Here we give some remarks about the above result.

\begin{remark}\label{rmk1}
Our theorem does not require the smallness of the initial condition to the remainder, as in \eqref{R_initial}. The well-posedness of the solution is valid up to time $T$ given the condition that $|u_b|T^{1/2}<\delta_0$ and $\e < \e_0(T,C_0)\ll 1$. 

This first constraint originates from the leading homogeneous term $\Gamma(f_1,R)$ in \eqref{L_tilde}. While the non-homogeneous term in \eqref{r_bc} can be controlled with $\e^{5/2-c_0}$, the homogeneous term $\Gamma(f_1,R)$ lacks decay-in-time properties from the nature of the heat solution in \eqref{u_p}. Therefore, such a constraint arises solely from the boundary movement. Furthermore, our problem possesses no non-trivial stationary solutions. This is because the heat equation \eqref{u_p} admits only self-similar solutions. We further discuss the stationary aspect of this problem by examining stationary solutions to a one-dimensional problem in Appendix \ref{sec:stationary}. Consequently, we expect that the Rayleigh problem should be studied in a dynamical setting, and it remains a challenging problem to investigate the long-time behavior within a self-similar framework.

The second constraint originates from the lack of macroscopic dissipation in the domain $\mathbb{T}^2\times \mathbb{R}^+$. Through high-order Hilbert expansion, the non-homogeneous term $h$ \eqref{h_def} and the nonlinear operator $\Gamma(R,R)$ \eqref{remainder_eqn} have sufficient power of $\e$. We may choose $\e$ to be small enough to control the growth in time from the $L^2_{x,v}$ energy estimate and the nonlinear growth from $\Gamma(R,R)$. In consequence, the smallness condition is not required in \eqref{R_initial}. We refer to the detailed argument in Section \ref{sec:proof_thm}.

\end{remark}

\begin{remark}\label{rmk2}
We avoid singularity at $t=0$ by imposing an initial perturbation $\delta$ in \eqref{u_p}. In the rigorous justification of the hydrodynamic limit, due to the Hilbert expansion \eqref{order_0} - \eqref{order_i}, the remainder estimate to \eqref{remainder_eqn} requires regularity estimates for the background fluid solution. However, these estimates introduce an initial singularity: $\p_{x_3} u_p \sim \frac{1}{\sqrt{t}}$ when $\delta=0$. Such difficulty does not arise if one studies the initial boundary value problem when Knudsen number $\e=1$ in \eqref{eqn_F}, since regularity estimate is not needed, see also \cite{kuo2010initial}.

In \cite{maekawa2021gevrey}, Maekawa established the stability of the Rayleigh boundary layer in the vanishing viscosity limit for Gevrey-class solutions with a singularity at 
$t=0$. While our work constructs kinetic solutions in Sobolev spaces, we expect that Gevrey regularity may be necessary to fully resolve the initial singularity. We leave the investigation of this issue, along with the vanishing viscosity problem, for future work.
\end{remark}

\begin{remark}\label{rmk:higher_order_expansion}
High regularity are required for the initial condition \eqref{f_i_intial} since the derivative of $f_i$ appears while solving \eqref{u_theta_i_eqn} and \eqref{f_i}. We refer to a more detailed discussion in Remark \ref{rmk:fluid_estimate}.
\end{remark}

\subsection{Outline}
In Section \ref{sec:prelim}, we list several properties of the collision operator and the background fluid solution. In Section \ref{sec:linear}, we establish the $L^\infty$ estimates for the linear problem. In Section \ref{sec:nonlinear}, we employ the linear estimate to study the nonlinear remainder equation of $R$ and conclude Theorem \ref{thm:well-posedness} in a classical perturbation framework. In Appendix \ref{sec:stationary}, we discuss the possibility of solving the steady problem with the same boundary conditions.

\subsection{Notation}\label{sec:notation}
Recall $\O=\mathbb{T}^2\times \mathbb{R}^+$. We use the general norms:
\begin{align*}
    &  \Vert f\Vert_{L^2_\nu} := \Vert \nu^{1/2}f(v)\Vert_{L^2_v}=\Big(\int_{\R^3}\nu(v)|f(v)|^2\dd v\Big)^{1/2}, \\
         & \Vert f\Vert_{L^2_{t,x,v}} := \Big(\int_0^t \Vert f(s)\Vert_{L^2_{x,v}}^2 \dd s \Big)^{1/2},\\
          & \Vert f\Vert_{L^2_{t,x,\nu}} := \Big(\int_0^t \Vert \nu^{1/2}f(s)\Vert_{L^2_{x,v}}^2 \dd s \Big)^{1/2},\\
    & \Vert f\Vert_{L^\infty_{t,x,v}} := \sup_{0\leq s\leq t}\sup_{(x,v)\in \O\times \mathbb{R}^3}|f(s,x,v)|, \\
    & \Vert f\Vert_{L^\infty_t L^2_{x,v}} := \sup_{0\leq s\leq t} \Vert f(s)\Vert_{L^2_{x,v}}.
    \end{align*}
We also use the following notations for the boundary integral:
\begin{align*}
    & \p\O : = \mathbb{T}^2\times \{x_3=0\}, \\
    & \gamma_\pm := \{(x,v)\in \p\O\times \mathbb{R}^3:  \ v_3\lessgtr 0 \}, \\
    & \int_{\gamma} f \dd \gamma := \int_{\gamma_+} f \dd\gamma + \int_{\gamma_-} f\dd \gamma, \\
    & \int_{\gamma_\pm} f \dd \gamma:= \int_{\p\O} \int_{v_3\lessgtr 0} f v_3 \dd v \dd S_x, \\
    & | f|_{L^2_{\gamma_\pm}} := \Big(\int_{\p\O}\int_{v_3\lessgtr 0} |f(x,v)|^2 |v_3| \dd v \dd S_x\Big)^{1/2}, \\
    & |f|_{L^\infty_{x,v}}:= \sup_{(x,v)\in\p\O \times \mathbb{R}^3} |f(x,v)|.
\end{align*}
Here $\dd S_x$ represents the surface integral.
    
Moreover, $f \lesssim g$  means that there exists $C>1$ such that $f\leq C g$, and $f\leq o(1)g$ and $f \lesssim o(1)g$ both mean that there exists $0<c\ll 1$ such that $f\leq c g$.

\section{Preliminary}\label{sec:prelim}

\subsection{Estimate to wall Maxwellian and Boltzmann operator}

Recall \eqref{wall_maxwell}, \eqref{def.bv} and \eqref{def.mu}. We have

\begin{lemma}\label{lemma:taylor}
The following properties hold for the wall Maxwellian \eqref{wall_maxwell}:
\begin{align}
    &   \Big|\mu^{-1/4} \frac{M_w - c_\mu \mu}{\sqrt{\mu}}  \Big| \lesssim \e ,. \label{taylor_first}
\end{align}

\end{lemma}

\begin{proof}
By the definition of $M_w$ in \eqref{wall_maxwell}, we apply the mean value theorem to have
\begin{align*}
    &  |M_w - c_\mu \mu| = \frac{1}{2\pi}\Big|  \exp\Big( - \frac{|v-(\e u_b,0,0)|^2}{2} \Big) - e^{-\frac{|v|^2}{2}}\Big| \\
    & \leq \frac{\e u_b}{2\pi} e^{-\frac{|v_2|^2+|v_3|^2}{2}} \sup_{|c|\leq \e u_b} e^{-\frac{|v-c|^2}{2}} \lesssim \frac{1}{2\pi} e^{-\frac{|v_2|^2 + |v_3|^2}{2}} e^{-\frac{3|v_1|^2}{8}} \lesssim \mu^{3/4}.
\end{align*}
In the second last inequality, from $\e u_b \ll 1$, we have
\begin{align}
    &  \sup_{|c|\leq \e u_b} e^{-\frac{|v_1-c|^2}{2}} \lesssim e^{-\frac{|v_1|^2}{2}}e^{\frac{|v_1|^2}{8}} e^{10|c|^2} \lesssim e^{-\frac{3|v_1|^2}{8}}. \label{extra_term}
\end{align}
This concludes \eqref{taylor_first}.

\end{proof}

The following estimates of the nonlinear operator $\Gamma$ are standard; we refer readers to \cite{EGKM} for the proof.
\begin{lemma}\label{lemma:gamma}
We have the following properties of the non-linear Boltzmann operator $\Gamma$ defined in \eqref{operator}:
\begin{align}
    & \Vert \nu^{-1/2}\Gamma(f,(\mathbf{I}-\mathbf{P})g)\Vert_{L^2_{x,v}} \lesssim \Vert wf\Vert_{L^\infty_{x,v}} \Vert (\mathbf{I}-\P)g\Vert_{L^2_{x,\nu}}, \notag\\
    & \Vert \nu^{-1/2}\Gamma((\mathbf{I}-\mathbf{P})f,g)\Vert_{L^2_{x,v}}\lesssim \Vert wg\Vert_{L^\infty_{x,v}} \Vert (\mathbf{I}-\mathbf{P})f\Vert_{L^2_{x,\nu}},\notag\\
    & \Vert \nu^{-1/2}\Gamma(\mathbf{P}f,\mathbf{P}g)\Vert_{L^2_{x,v}} \lesssim \Vert  f\Vert_{L^2_{x,v}} \Vert w g\Vert_{L^\infty_{x,v}}, \label{gamma_f_g_2infty}\\
    & \Vert \langle v\rangle^{-1} w \Gamma(f,g)\Vert_{L^\infty_{x,v}} \lesssim \Vert wf\Vert_{L^\infty_{x,v}}\Vert wg\Vert_{L^\infty_{x,v}}, \label{gamma_infty}\\
    & \Vert \Gamma(\mu^{1/4} f, \mu^{1/4} g)\Vert_{L^2_{x,v}} \lesssim \Vert f\Vert_{L^2_{x,v}}\Vert g\Vert_{L^\infty_{x,v}}.\label{gamma_mu_f}
\end{align}
\end{lemma}

\subsection{Fluid estimate}
We summarize the fluid estimate in the following lemmas.

\begin{lemma}\label{lemma:fluid}
Recall the background fluid solution in \eqref{u_p}. The following estimates to $u$ holds:
\begin{align}
    & \Vert u(t) \Vert_{L^2_x} \lesssim u_b (t+\delta)^{1/4} , \label{add.u2x} 
\end{align}
for $1\leq j+k\leq 6$,
\begin{align}
    &\Vert \p_t^j \nabla_x^k u \Vert_{L^2_x}  \lesssim  u_b ,    \label{u_l2_est}
\end{align}
\begin{align}
    & \Vert u\Vert_{L^\infty_{x}}+  \Vert \p_t^j \nabla_x^k u\Vert_{L^\infty_x} \lesssim u_b. \label{u_linfty_est}
\end{align}
Thus, we have the following control for $f_1$ defined in \eqref{f1}:
\begin{align}
    &\sum_{j+k\leq 7}(\Vert w\p_t^j\nabla_x^kf_1\Vert_{L^\infty_{x,v}} + \Vert \p_t^j\nabla_x^k f_1\Vert_{L^2_{x,v}} )\lesssim  u_b. \label{f_12_linfty_bdd}
\end{align}

\end{lemma}

\begin{proof}
By the definition of $u$ in \eqref{u_p}, we directly have
\begin{align*}
    & \Vert u\Vert_{L^\infty_{x}} \lesssim u_b.
\end{align*}

To compute $\Vert u\Vert_{L^2_x}$, we integrate $u_p$ as
\begin{align*}
    &\Vert u_p\Vert_{L^2_x}^2  \lesssim |u_b|^2 \int_0^\infty   \Big| \int^\infty_{\frac{x_3}{\sqrt{4\kappa (t+\delta)}}} e^{-r+1} \dd r   \Big|^2 \dd x_3  \\
    &\lesssim \int_0^\infty     \exp\Big\{-\frac{2x_3}{\sqrt{4\kappa (t+\delta)}} \Big\}       \dd x_3 \lesssim |u_b|^2\sqrt{t+\delta}.
\end{align*}
This leads to the $L^2$ control of $u$:
\begin{align*}
    & \Vert u(t)\Vert_{L^2_x} \lesssim (t+\delta)^{1/4} u_b,
\end{align*}
which proves \eqref{add.u2x}.

Next, we estimate $\p_t u$. We compute that
\begin{align*}
    & \p_t u_p = \Big|\frac{-2u_b}{\sqrt{\pi}} \frac{x_3}{\sqrt{4\kappa}(t+\delta)^{3/2}} \exp\Big(-\frac{|x_3|^2}{4\kappa (t+\delta)}\Big)\Big|\lesssim \frac{u_b}{t+\delta} \lesssim u_b,
\end{align*}
\begin{align*}
    & \Vert \p_t u_p\Vert_{L^2_x}^2 \lesssim  \Big\Vert \frac{u_b x_3}{\sqrt{4\kappa}(t+\delta)^{3/2}} \exp\Big(-\frac{|x_3|^2}{4\kappa (t+\delta)}\Big)  \Big\Vert_{L^2_x}^2 \\
    &\lesssim \frac{|u_b|^2}{(t+\delta)^2}\int_0^\infty  \frac{|x_3|^2}{t+\delta} \exp\Big(-\frac{|x_3|^2}{2\kappa(t+\delta)} \Big) \dd x_3  \\
    &\lesssim \frac{|u_b|^2\sqrt{t+\delta}}{(t+\delta)^2}  \lesssim \frac{|u_b|^2}{(t+\delta)^{3/2}} \lesssim |u_b|^2.
\end{align*}
This leads to
\begin{align*}
    & \Vert \p_t u\Vert_{L^\infty_x}\lesssim \frac{u_b}{t+\delta} \lesssim u_b, \  \Vert \p_t u\Vert_{L^2_x} \lesssim \frac{u_b}{(t+\delta)^{3/4}} \lesssim u_b.
\end{align*}
Similarly, the high order time derivative can be computed as
\begin{align*}
    &  \Vert\p_t^j u_p\Vert_{L^\infty_x} + \Vert\p_t^j u_p\Vert_{L^2_x} \lesssim u_b.
\end{align*}

Next, we compute the spatial derivative of $u$. We have
\begin{align*}
    & |\p_{x_3} u_p| = \Big| \frac{2u_b}{\sqrt{\pi}} \frac{1}{\sqrt{4\kappa (t+\delta)}} \exp\Big( -\frac{|x_3|^2}{4\kappa (t+\delta)} \Big) \Big|  \lesssim \frac{u_b}{\sqrt{(t+\delta)}}\lesssim u_b ,
\end{align*}
\begin{align*}
    & \Vert \nabla_x u_p\Vert_{L^2_x}^2 \lesssim   \Big\Vert    \frac{u_b}{\sqrt{4\kappa (t+\delta)}} \exp\Big\{-\frac{|x_3|^2}{4\kappa(t+\delta)} \Big\}   \Big\Vert_{L^2_x}^2 \\
    &\lesssim \frac{|u_b|^2}{t+\delta} \int_0^\infty \exp\Big(-\frac{2|x_3|^2}{4\kappa (t+\delta)} \Big) \dd x_3 \\
    & \lesssim |u_b|^2\frac{\sqrt{t+\delta}}{t+\delta} = \frac{|u_b|^2}{\sqrt{t+\delta}} \lesssim |u_b|^2.
\end{align*}
This leads to
\begin{align*}
    &  \Vert \nabla_x u\Vert_{L^\infty_x} \lesssim \frac{u_b}{(t+\delta)^{1/2}} \lesssim u_b, \    \Vert \nabla_x u\Vert_{L^2_x} \lesssim \frac{u_b}{(t+\delta)^{1/4}} \lesssim u_b.
\end{align*}

The estimate to the high order spatial derivative $\nabla^k_x u$ and the mixed derivative $\p_t^j \nabla_x^k u$ are similar, as we have
\begin{align*}
    &  |\p_{x_3}^2 u_p| \lesssim \frac{u_b}{(t+\delta)} \frac{x_3}{\sqrt{4\kappa (t+\delta)}} \exp \Big(-\frac{|x_3|^2}{4\kappa (t+\delta)} \Big) \lesssim \frac{u_b}{t+\delta} \lesssim u_b, \\
    & |\p_t \p_{x_3} u_p|  \lesssim u_b\Big[\frac{1}{(t+\delta)^{3/2}}  + \frac{1}{(t+\delta)^{3/2}} \frac{|x_3|^2}{t+\delta} \Big] \exp \Big(-\frac{|x_3|^2}{4\kappa (t+\delta)} \Big) \\
    &\lesssim \frac{u_b}{(t+\delta)^{3/2}} \lesssim u_b,
\end{align*}
and inductively, we conclude \eqref{u_l2_est} and \eqref{u_linfty_est}.

We conclude the lemma.
\end{proof}

\begin{lemma}\label{lemma:inter_solution}
Assume the initial condition satisfies \eqref{f_i_intial} and the compatibility condition. Then it holds
\begin{align}
    &    \sum_{i=2}^5 \sum_{j+k\leq 8-i} \Vert \mu^{-1/4}\p_t^j \nabla_{x}^kf_i\Vert_{L^2_{x,v}} +  \sum_{i=2}^5 \sum_{j+k\leq 6-i}\Vert \mu^{-1/4} \p_t^j\nabla_{x}^k f_i\Vert_{L^\infty_{x,v}}  \notag\\
    &+ \sum_{i=2}^4 \sum_{j+k\leq 1} \Vert \mu^{-1/4}we^{c\xi} \p_t^j \nabla^{k}_{x_\parallel}\mathfrak{f}_i\Vert_{L^\infty_{\xi,v}} < C_3(T,\delta)+u_b. \notag
\end{align}

\end{lemma}

\begin{remark}\label{rmk:fluid_estimate}
We refer to the proof in \cite{guo2021hilbert} for the incompressible Euler equation. Here we do not have the viscous boundary layer, and the equation is given by the linear Navier-Stokes-Fourier equation instead. Therefore, we expect to have the  high regularity control for the interior solution and Knudsen boundary layer solution, for any given time $T$(though upper bound depends on $T$). Also see \cite{wu2022hydrodynamic}.

There exhibits an extra gain in the velocity weight $\mu^{-1/4}$ since $(\mathbf{I}-\mathbf{P})f_i$ itself contains a Maxwellian factor $\mu^{1/2}$. This can be inductively justified using \eqref{f_i}.

Here we note that the $L^2_{x,v}$ estimate needs higher regularity with order $8-i$, then by the Sobolev embedding, the $L^\infty_{x,v}$ estimate enjoys regularity up to order $6-i$. Moreover, the requirement of the derivative order differs for $f_i$. For instance, $f_5$ requires one order of derivative in \eqref{h_def}. While solving \eqref{u_theta_i_eqn}, the source term \eqref{S_theta}, \eqref{S_u} requires the regularity for $\nabla_x\p_t f_3$ and $\nabla_x^2 f_4$. Therefore, two order of derivative to $f_4$ is needed. Inductively, we need $6-i$ order derivative of $f_i$ in the $L^\infty_{x,v}$ control. 

We also point out that the $u_b$ occurs from the contribution of $f_1$ in \eqref{f_12_linfty_bdd}, as $(\mathbf{I}-\mathbf{P})f_2$ is determined by $f_1,\nabla_x f_1$ from \eqref{f2_derivation}. Inductively, from \eqref{f_i} and \eqref{f_i_bl}, $f_i,\mathfrak{f}_i$ are dependent of $\p_t^j \nabla_x^k f_1$ with $j+k \leq i-1$. In consequence, the constant $C_3(T,\delta)$ depends on time $T$ and $\delta$ from $f_1$ in \eqref{f1}.
\end{remark}

Then the inhomogeneous source \eqref{h_def} can be controlled by the following lemma:
\begin{lemma}\label{lemma:h_control}
It holds
\begin{align*}
& \Vert h(t)\Vert_{L^2_{x,v}} + \Vert wh(t)\Vert_{L^\infty_{x,v}} \lesssim C_3(T,\delta) \e^{1/2-c_0}.    
\end{align*}

\end{lemma}

\begin{proof}
From Lemma \ref{lemma:inter_solution}, with the extra weight $\mu^{-1/4}$, the $v$ growth in \eqref{h_def} can be absorbed. It suffices to estimate the nonlinear terms. Again, with the extra weight $\mu^{-1/4}$, all nonlinear terms can be controlled using \eqref{gamma_mu_f} and \eqref{gamma_infty} in Lemma \ref{lemma:gamma}.
\end{proof}

\section{Linear estimate}\label{sec:linear}

In this section, we investigate the linear equation:
\begin{align}
\begin{cases}
     & \e \p_t f + v\cdot \nabla_x f + \frac{1}{\e} \mathcal{L}f = g, \\
    & f|_{\gamma_-} = P_\gamma f + q . 
\end{cases} \label{linear_f}   
\end{align}
The estimate in this section will be applied to the non-linear equation studied in Section \ref{sec:nonlinear}.

In this section, we mainly establish an $L^\infty$ estimate to $f$ in Lemma \ref{lemma:linfty}. We first recall the velocity weight defined in \eqref{w_weight}. In the following lemma, we construct the $L^\infty$ estimate to \eqref{linear_f}. Since the proof is standard from an $L^2-L^\infty$ bootstrap argument, we refer readers to \cite{EGKM2}.

\begin{lemma}\label{lemma:linfty}
Let $f$ be the solution to \eqref{linear_f}. Assume that
\begin{align*}
    & \Vert \e^{3/2}w f(t)\Vert_{L^\infty_{x,v}}  < \infty.
\end{align*}
Then it holds that
\begin{align}
  \Vert \e^{3/2} w f(t)\Vert_{L^\infty_{x,v}}  &  \lesssim \Vert \e^{3/2} w f_0\Vert_{L^\infty_{x,v}} +  \e^{3/2} |wq|_{L^\infty_{t,x,v}}  \notag\\
    &+ \e^{5/2} \Vert \langle v\rangle^{-1} w g\Vert_{L^\infty_{t,x,v}} + \Vert f\Vert_{L^\infty_t L^2_{x,v}}.  \notag
\end{align}

\end{lemma}

\section{Nonlinear estimate}\label{sec:nonlinear}
In this section, we focus on the nonlinear remainder equation \eqref{remainder_eqn}.

We denote the energy and energy dissipation norm as
\begin{align}
    &  \Vert R\Vert_{E,t} : = \e^{-1}\Vert (\mathbf{I}-\P)R\Vert_{L^2_{t,x,\nu}} + \e^{-1/2} |(I-P_\gamma)R|_{L^2_{t,\gamma_+}} +  \Vert R\Vert_{L^\infty_t L^2_{x,v}} , \label{energy_norm}
\end{align}
and the velocity weighted sup-norm as
\begin{align}
    &  \Vert R\Vert_{\infty,t} : =  \Vert \e^{3/2}wR\Vert_{L^\infty_{t,x,v}} . \notag
\end{align}

In this section, we will mainly prove the following a priori estimate.

\begin{proposition}\label{prop:norm_bdd}
Let $R$ be the solution to \eqref{remainder_eqn} with boundary condition \eqref{R_bc} on $[0,T]$ for $T>0$. Let $c_0=\frac{1}{4}.$ Suppose $t\leq T$, $|u_b|\sqrt{T}\ll 1$, and
\begin{align}
\Vert R\Vert_{E,t} + \Vert R\Vert_{\infty,t} < \infty.     \label{apriori_assumption}
\end{align}
Then the following estimate holds:
\begin{align*}
    & \Vert R\Vert_{E,t} + \Vert R\Vert_{\infty,t}  \\
    &\lesssim CT\e^{1/2}+ \Vert \e^{3/2}wR_0\Vert_{L^\infty_{x,v}} + \Vert R_0\Vert_{L^2_{x,v}}  +  T\e^{1/2}\Vert R\Vert_{E,t}^2 + T\e^{1/2}\Vert R\Vert_{\infty,t}^2.
\end{align*}
\end{proposition}

Now, to prove Proposition \ref{prop:norm_bdd} we investigate the nonlinear equation \eqref{remainder_eqn} by replacing $g$ and $q$ in \eqref{linear_f} with
\begin{align}
    & g = \e h + \tilde{\mathcal{L}}R + \e^{3/2+c_0}\Gamma(R,R), \label{g_nonlinear}\\
    & q = \frac{M_w-c_\mu \mu}{\sqrt{\mu}}\int_{v_3'<0} R \sqrt{\mu(v')}|v_3'| \dd v' + r \ \mbox{for } x_3=0, v_3>0,
    \label{q_nonlinear} 
\end{align}
where $r$ is defined in \eqref{r_bc}.

\subsection{Preparation: estimates of the boundary term and source term}

In this section, we summarize the necessary estimates for both the source term $g$ \eqref{g_nonlinear} and boundary term $q$ \eqref{q_nonlinear}.

The first term \eqref{q_nonlinear} suggests a trace control in the following lemma.

\begin{lemma}[Ukai trace Theorem, Lemma 3.2 of \cite{EGKM2}]\label{lemma:ukai}
Define the non-grazing set as $\gamma_+^{\delta_1}:= \{(x,v)\in \gamma_+: |v_3|>\delta_1 \text{ and }|v|<\frac{1}{\delta_1}\}$ for $\delta_1>0$. Then for $0<\delta_1\ll 1$, it holds that
\begin{align*}
    & \frac{1}{\e} \int_0^t \int_{\gamma_+^{\delta_1}} |f|\dd \gamma \dd s \lesssim_{\delta_1} \int_{\mathbb{R}^3} \int_{\O} |f(0)| \dd x \dd v + \int_0^t \int_{\mathbb{R}^3} \int_{\O} |f| \dd x \dd v \dd s \\
    &+ \int_0^t \int_{\mathbb{R}^3} \int_{\O} |\p_t f + \frac{1}{\e}v\cdot \nabla_x f| \dd x \dd v \dd s.
\end{align*}

\end{lemma}

We estimate the contribution of the first term in \eqref{q_nonlinear} in the following lemma.

\begin{lemma}\label{lemma:trace_R}
Let $R$ be the solution to \eqref{remainder_eqn}, then the following trace control holds:
\begin{align}
    & \e^{-1} \int_0^t \int_{\gamma_+} \Big|\frac{M_w - c_\mu \mu}{\sqrt{\mu}} \int_{v_3'<0} R \sqrt{\mu(v')} |v_3'| \dd v'  \Big|^2 \dd \gamma \dd s \notag\\
    &\lesssim \e \int_0^t \int_{\p\O} \int_{v_3' < 0 } |R|^2 |v_3'| \dd v'  \dd S_x \dd s \notag\\
    &\lesssim \e \int_0^t\int_{\gamma_+} |(I-P_\gamma)R|^2 \dd \gamma \dd s + \e^2\int_{\mathbb{R}^3} \int_{\O} |R_0|^2 \dd x \dd v \notag\\
    &+ \e^2 \int_0^t \int_{\mathbb{R}^3} \int_{\O} |R|^2 \dd x \dd v\dd s \notag\\
    & + \Big|\int_0^t \int_{\mathbb{R}^3} \int_{\O}-\mathcal{L}(R)R + \e^2 hR + \e \tilde{\mathcal{L}}(R)R + \e^{5/2+c_0}\Gamma(R,R)R   \dd x \dd v \dd s  \Big| 
  .\label{trace_R}
\end{align}

\end{lemma}

\begin{proof}

We apply \eqref{taylor_first} to compute that
\begin{align*}
    & \e^{-1} \int_{\gamma_+} \Big|\frac{M_w - c_\mu \mu}{\sqrt{\mu}} \int_{v_3'<0}R\sqrt{\mu(v')} |v_3'| \dd v' \Big|^2 \dd \gamma \\
    &\lesssim \e \int_{\gamma_+} \mu^{1/4}(v) \Big| \int_{v_3' < 0 } R \sqrt{\mu(v')} |v_3'| \dd v' \Big|^2 \dd\gamma \\
    & \lesssim \e \int_{\p\O} \Big| \int_{v_3' < 0 } R \sqrt{\mu(v')} |v_3'| \dd v' \Big|^2 \dd S_x \lesssim \e \int_{\p\O} \int_{v_3'<0} |R|^2 |v_3'|\dd v' \dd S_x \\ 
    &\leq \e \int_{\p\O}  \Big[ \int_{|v_3'|<\delta_1 \text{ or }|v'|>\delta_1^{-1}} + \int_{|v_3'|>\delta_1 , |v'|<\delta^{-1}}\Big] |R|^2  |v_3'| \dd v'  \dd S_x \\
    & \lesssim  \e \int_{\p\O} \int_{|v_3'|<\delta_1 \text{ or }|v'|>\delta_1^{-1}} [|(I-P_\gamma) R|^2 + |P_\gamma R|^2 ]  |v_3'| \dd v' \dd S_x + \e \int_{\gamma_+^{\delta_1}} |R|^2 \dd \gamma \\
    & \lesssim \e \int_{\gamma_+} |(I-P_\gamma)R|^2 \dd \gamma + \e \int_{\gamma_+^{\delta_1}} |R|^2 \dd \gamma \\
    &+ \e \int_{\p\O} \int_{|v_3|<\delta_1 \text{ or }|v|>\delta_1^{-1}}  \mu(v) |v_3| \Big|\int_{v_3'<0} R \sqrt{\mu(v')}|v_3'| \dd v'\Big|^2 \dd v \dd S_x \\
    &\lesssim \e \int_{\gamma_+} |(I-P_\gamma)R|^2 \dd \gamma + \e \int_{\gamma_+^{\delta_1}} |R|^2 \dd \gamma + o(1)\e \int_{\p\O} \int_{v_3'<0} |R|^2 |v_3'| \dd v'  \dd S_x.
\end{align*}
Absorbing the last term by the second line, we further derive that
\begin{align*}
    & \e^{-1} \int_{\gamma_+} \Big|\frac{M_w - c_\mu \mu}{\sqrt{\mu}} \int_{v_3'<0}R\sqrt{\mu(v')} |v_3'| \dd v' \Big|^2 \dd \gamma \\
    & \lesssim \e \int_{\p\O} \int_{v_3' < 0 } |R|^2 |v_3'| \dd v'  \dd S_x  \lesssim \e \int_{\gamma_+} |(I-P_\gamma)R|^2 \dd \gamma + \e^2 \frac{1}{\e} \int_{\gamma_+^{\delta_1}} |R|^2 \dd \gamma .
\end{align*}
From \eqref{remainder_eqn}, the equation of $R^2$ satisfies
\begin{align*}
    \p_t R^2 + \frac{1}{\e} v\cdot \nabla_x R^2 + \frac{2}{\e^2} \mathcal{L}(R) R = 2h R + \frac{2\tilde{\mathcal{L}}(R)R}{\e} + 2\e^{1/2+c_0}\Gamma(R,R)R.
\end{align*}
Applying Lemma \ref{lemma:ukai} to $\frac{1}{\e}\int_{\gamma_+^{\delta_1}} |R|^2 \dd \gamma$, we conclude \eqref{trace_R}.

\end{proof}

In view of Lemma \ref{lemma:linfty} and the $L^2$ energy estimate, we need the following two estimates of $q$ defined in \eqref{q_nonlinear}:
\begin{align*}
    &   \ |\e^{3/2}wq|_{L^\infty_{t,x,v}} , \ \ \ |q|_{L^2_{\gamma_-}}.
\end{align*}

Then we summarize all estimates for $q$ \eqref{q_nonlinear} in the following lemma.

\begin{lemma}\label{lemma:q_estimate}
The contribution of $q$ in Lemma \ref{lemma:linfty} is bounded as
\begin{align}
    & |\e^{3/2}wq|_{L^\infty_{t,x,v}} \lesssim \e \Vert \e^{3/2}wR\Vert_{L^\infty_{t,x,v}} + C_3\e^{4-c_0}.\label{q_infty}
\end{align}

While the first term of \eqref{q_nonlinear} is estimated in Lemma \ref{lemma:trace_R}, the second term of \eqref{q_nonlinear} is bounded as
\begin{align}
    & \int_{\gamma_-} r^2 \dd \gamma\lesssim C_3\e^{5-2c_0} .  \label{r_l2} 
\end{align}

\end{lemma}

\begin{proof}
First, we derive an upper bound for $r$ defined in \eqref{R_bc}. By Lemma \ref{lemma:inter_solution}, the first two terms in \eqref{R_bc} are bounded as
\begin{align*}
    & \e^{5/2-c_0}\Big|\mu^{-1/4} \Big[ \frac{M_w}{\sqrt{\mu}} \int_{v_3'<0} f_5 \sqrt{\mu(v')}|v_3'| \dd v'  -f_5\Big] \Big|_{L^\infty_{t,x,v}} \lesssim C_3\e^{5/2-c_0} .
\end{align*}

The last term in \eqref{r_bc} is also controlled as $C_3 \e^{5/2}$. Thus we obtain
\begin{align}
    & |w r|_{L^\infty_{t,x,v}} \lesssim  C_3\e^{5/2-c_0} .  \label{wr_infty}
\end{align}

Next, we prove \eqref{q_infty}. We apply \eqref{taylor_first} to compute that for $w\lesssim \mu^{-1/4}$,
\begin{align*}
    & \Big|\e^{3/2} e^{\beta |v|^2} \frac{M_w-c_\mu \mu}{\sqrt{\mu}} \int_{v_3'<0} R\sqrt{\mu(v')}|v_3'| \dd v'  \Big|_{L^\infty_{t,x,v}}  \lesssim \e \Vert \e^{3/2} wR\Vert_{L^\infty_{t,x,v}}.
\end{align*}
For $\e^{3/2}r$ we apply \eqref{wr_infty} to have
\begin{align*}
    & |\e^{3/2} w  r|_{L^\infty_{t,x,v}} \lesssim \e^{3/2} |w r|_{L^\infty_{t,x,v}}\lesssim C\e^{4-c_0} .
\end{align*}
We conclude \eqref{q_infty}.

Last, we prove \eqref{r_l2}. We apply \eqref{wr_infty} to compute that
\begin{align*}
    &  \int_{\gamma_-} r^2 \dd \gamma \lesssim |w r|_{L^\infty_{x,v}}^2 \int_{\gamma_-} w^{-2} \dd \gamma \lesssim C \e^{5-2c_0}  .
\end{align*}

We have completed the proof.
\end{proof}

In the following lemma, we summarize the estimate of the source term in \eqref{remainder_eqn}.

\begin{lemma}\label{lemma:source_estimate}
We control the source term in \eqref{remainder_eqn} under $L^2$ energy estimate:
\begin{align}
    & \Vert h \Vert_{L^2_{x,v}} \lesssim C_3\e^{1/2-c_0} ,   \label{h_l2_bdd_2}  
\end{align}

\begin{equation}\label{l_tilde_2}
\Vert \nu^{-1/2} \tilde{\mathcal{L}}(R) \Vert_{L^2_{x,v}} \lesssim (u_b+C_3\e) \Vert R\Vert_{L^2_{x,v}} + \e (u_b+C_3\e) \Vert \e^{-1}(\mathbf{I}-\mathbf{P})R\Vert_{L^2_{x,\nu}} .
\end{equation}
\begin{align}
 \Vert \nu^{-1/2}\Gamma(R,R)\Vert_{L^2_{x,v}}   &  \lesssim \e^{-1/2} \Vert \e^{-1}(\mathbf{I}-\P)R\Vert_{L^2_{x,\nu}} \Vert \e^{3/2}wR\Vert_{L^\infty_{x,v}} \notag \\
    & \ \ \ \ \ + \e^{-3/2}\Vert R\Vert_{L^2_{x,v}}\Vert \e^{3/2}wR\Vert_{L^\infty_{x,v}}, \label{gamma_l2_bdd}
\end{align}

We also have the $L^\infty_{x,v}$ control to the source term:
\begin{align}
    & \Vert \langle v\rangle^{-1} w h\Vert_{L^\infty_{x,v}} \lesssim \e^{1/2-c_0} C_3,  \label{h_infty_bdd_2}
\end{align}
\begin{equation}\label{l_tilde_infty_bdd}
 \Vert \langle v\rangle^{-1} w \tilde{\mathcal{L}}(R)\Vert_{L^\infty_{x,v}} \lesssim (u_b+C_3\e) \e^{-3/2} \Vert \e^{3/2}wR\Vert_{L^\infty_{x,v}} .
\end{equation}
\begin{equation}\label{gamma_infty_bdd}
     \Vert \langle v\rangle^{-1}w\Gamma(R,R)\Vert_{L^\infty_{x,v}} \lesssim \e^{-3}\Vert \e^{3/2}wR\Vert_{L^\infty_{x,v}}^2 .
\end{equation}

\end{lemma}

\begin{proof}
First of all, \eqref{h_l2_bdd_2} is proved in Lemma \ref{lemma:h_control}.

For \eqref{l_tilde_2}, we apply Lemma \ref{lemma:gamma} to have
\begin{align*}
    & \Vert \nu^{-1/2} \tilde{\mathcal{L}}(R) \Vert_{L^2_{x,v}} = 2\Vert \nu^{-1/2}\Gamma(f_1+\sum_{i=2}^4 \e^{i-1} (f_i+\mathfrak{f}_i)+\e^4 f_5,R) \Vert_{L^2_{x,v}} \\
    & \lesssim \Vert \nu^{-1/2}[\Gamma(f_1+\sum_{i=2}^4 \e^{i-1} (f_i+\mathfrak{f}_i)+\e^4 f_5, (\mathbf{I}-\P)R) \Vert_{L^2_{x,v}}  \\
    &+ \Vert \nu^{-1/2}\Gamma(f_1+\sum_{i=2}^4 \e^{i-1} (f_i+\mathfrak{f}_i)+\e^4 f_5, \P R) \Vert_{L^2_{x,v}} \\
    &\lesssim \Vert f_1+\sum_{i=2}^4 \e^{i-1} (f_i+\mathfrak{f}_i)+\e^4 f_5 \Vert_{L^\infty_{x,v}} \Vert \P R\Vert_{L^2_{x,v}}  \\
    &+ \Vert f_1+\sum_{i=2}^4 \e^{i-1} (f_i+\mathfrak{f}_i)+\e^4 f_5\Vert_{L^\infty_{x,v}} \Vert (\mathbf{I}-\P)R\Vert_{L^2_{x,\nu}}  \\
& \lesssim \Vert f_1+\sum_{i=2}^4 \e^{i-1} (f_i+\mathfrak{f}_i)+\e^4 f_5 \Vert_{L^\infty_{x,v}} \Vert  R\Vert_{L^2_{x,v}}   \\
&+ \e\Vert f_1+\sum_{i=2}^4 \e^{i-1} (f_i+\mathfrak{f}_i)+\e^4 f_5\Vert_{L^\infty_{x,v}} \e^{-1}\Vert (\mathbf{I}-\P)R\Vert_{L^2_{x,\nu}} \\
& \lesssim (u_b + C\e) \Vert R\Vert_{L^2_{x,v}} + \e u_b \Vert \e^{-1}(\mathbf{I}-\mathbf{P})R\Vert_{L^2_{x,\nu}}.
\end{align*}
In the last line, we applied Lemma \ref{lemma:fluid} and Lemma \ref{lemma:inter_solution}. We conclude \eqref{l_tilde_2}.

For \eqref{gamma_l2_bdd}, we apply Lemma \ref{lemma:gamma} to have,
\begin{align*}
    & \Vert \nu^{-1/2}\Gamma(R,R)\Vert_{L^2_{x,v}} \lesssim  \Vert \nu^{-1/2}\Gamma(R, (\mathbf{I}-\P)R) \Vert_{L^2_{x,v}}  \\
    &  + \Vert \nu^{-1/2}\Gamma((\mathbf{I}-\P)R, R) \Vert_{L^2_{x,v}} + \Vert \nu^{-1/2}\Gamma(\P R, \P R) \Vert_{L^2_{x,v}} \\
&   \lesssim  \e^{1/2}\Vert \e^{-1}(\mathbf{I}-\mathbf{P})R\Vert_{L^2_{x,\nu}} \Vert \e^{1/2}wR\Vert_{L^\infty_{x,v}}+ \Vert \nu^{-1/2} \Gamma(\P R , \P R)\Vert_{L^2_{x,v}} .
\end{align*}

For the last term, we apply \eqref{gamma_f_g_2infty} to have
\begin{align*}
    & \Vert \nu^{-1/2}\Gamma(\P R, \P R)\Vert_{L^2_{x,v}} \lesssim \e^{-3/2}\Vert R\Vert_{L^2_{x,v}} \Vert \e^{3/2} wR\Vert_{L^\infty_{x,v}} .
\end{align*}
This concludes \eqref{gamma_l2_bdd}.

Last, we move onto the $L^\infty$ estimate. \eqref{h_infty_bdd_2} has been proved in Lemma \ref{lemma:h_control}. \eqref{l_tilde_infty_bdd} and \eqref{gamma_infty_bdd} directly follow by applying \eqref{gamma_infty} and the estimates for $f_i,\mathfrak{f}_i$ in Lemma \ref{lemma:fluid} and Lemma \ref{lemma:inter_solution}.
\end{proof}

\subsection{Nonlinear energy estimate}

To prove Proposition \ref{prop:norm_bdd}, we construct the $L^2$ energy estimate and the $L^\infty$ estimate, in the following Lemma \ref{lemma:l2} and Lemma \ref{lemma:linfty_R} respectively.

\begin{lemma}\label{lemma:l2}
Under the assumptions Proposition \ref{prop:norm_bdd}, with $c_0=\frac{1}{4}$, we have the following control for the energy and dissipation \eqref{energy_norm}:
    \begin{align}
& \Vert R\Vert_{E,t}\lesssim C_3T\e^{1/2} + \Vert R_0\Vert_{L^2_{x,v}}^2 + T\e^{1/2}\Vert R\Vert_{E,t}\Vert R\Vert_{\infty,t}. \label{l2_all_bdd}
    \end{align}
\end{lemma}

\begin{proof}

We multiply \eqref{remainder_eqn} by $R$ and take integration in $[0,t]\times \O\times \mathbb{R}^3$ to obtain
\begin{align}
    & \Vert R(t)\Vert^2_{L^2_{x,v}} + \e^{-1} \int_0^t\int_{\gamma_+} |(I-P_\gamma)R|^2 \dd \gamma \dd s + \e^{-2} \Vert (\mathbf{I}-\P)R \Vert^2_{L^2_{t,x,\nu}}  \notag\\
    & \lesssim \Vert R_0\Vert_{L^2_{x,v}}^2+ \underbrace{\e^{-1} \int_0^t\int_{\gamma_-} r^2 \dd \gamma \dd s}_{\eqref{L2_energy}_1} \notag \\
    &+ \underbrace{\e^{-1} \int_0^t \int_{\gamma_+} \Big|\frac{M_w - c_\mu \mu}{\sqrt{\mu}} \int_{v_3'<0}R\sqrt{\mu(v')} |v_3'| \dd v' \Big|^2 \dd \gamma \dd s}_{\eqref{L2_energy}_2}  \notag\\
    & + \underbrace{\Big|\int_0^t\int_{\mathbb{R}^3} \int_{\O} \e^{-1}\tilde{\mathcal{L}}(R) R \dd x \dd v \dd s\Big|}_{\eqref{L2_energy}_3} + \underbrace{  \Big|\int_0^t\int_{\mathbb{R}^3} \int_{\O}  h R \dd x \dd v \dd s\Big|}_{\eqref{L2_energy}_4}  \notag\\
    &+ \underbrace{\Big|\int_0^t\int_{\mathbb{R}^3} \int_{\O}  \e^{3/4} \Gamma(R,R) R \dd x \dd v \dd s \Big|}_{\eqref{L2_energy}_5} . \label{L2_energy}
\end{align}

By \eqref{r_l2} in Lemma \ref{lemma:q_estimate}, we have
\begin{align}
    &    \eqref{L2_energy}_1 \lesssim C_3T\e^{3}. \notag
\end{align}

For the contribution of the boundary term, by Lemma \ref{lemma:trace_R}, we have
\begin{align}
    & \eqref{L2_energy}_2\lesssim \e \int_0^t \int_{\gamma_+} |(I-P_\gamma)R|^2 \dd \gamma \dd s  \notag\\
    &+ \e^2 \int_{\mathbb{R}^3} \int_{\O} |R_0|^2 \dd x \dd v + \e^2 \int_0^t \int_{\mathbb{R}^3} \int_{\O} |R|^2 \dd x \dd v \dd s \notag\\
    & + \int_0^t \int_{\mathbb{R}^3} \int_{\O} \big|-\mathcal{L}(R)R + \e^2 hR + \e \tilde{\mathcal{L}}(R)R + \e^{11/4}\Gamma(R,R)R \big| \dd x \dd v \dd s \notag\\
    & \lesssim \e |(I-P_\gamma)R|^2_{L^2_{t,\gamma_+}}  +  \e^2 \Vert R_0\Vert_{L^2_{x,v}}^2 + \e^2 T  \Vert R\Vert_{L^\infty_t L^2_{x,v}}^2 \notag\\
    & +  \Vert  (\mathbf{I}-\P)R\Vert_{L^2_{t,x,\nu}}^2  + \e^{2}[\eqref{L2_energy}_3 + \eqref{L2_energy}_4 + \eqref{L2_energy}_5]. \notag
\end{align}

It remains to control $\eqref{L2_energy}_3 , \eqref{L2_energy}_4, \eqref{L2_energy}_5$. By definition of $\tilde{\mathcal{L}}$ in \eqref{L_tilde} and applying \eqref{l_tilde_2} in Lemma \ref{lemma:source_estimate}, we have
\begin{align}
    & \eqref{L2_energy}_3 = \Big| \int_0^t\int_{\mathbb{R}^3} \int_{\O} \e^{-1}\tilde{\mathcal{L}}(R) (\mathbf{I}-\mathbf{P})R \dd x \dd v \dd s  \Big| \notag\\
    &\lesssim \int_0^t \big[\Vert \nu^{-1/2}\tilde{\mathcal{L}}(R)\Vert_{L^2_{x,v}}^2 + o(1)\Vert \e^{-1}(\mathbf{I}-\mathbf{P})R\Vert_{L^2_{x,\nu}}^2 \big] \dd s \notag\\
    & \lesssim  T|u_b+C_3\e|^2 \Vert R\Vert_{L^\infty_t L^2_{x,v}}^2 + \e^2 |u_b+C_3\e|^2 \Vert \e^{-1}(\mathbf{I}-\P)R\Vert_{L^2_{t,x,\nu}}^2  \notag \\
    &+  o(1)\Vert \e^{-1}(\mathbf{I}-\P)R\Vert_{L^2_{t,x,\nu}}^2. \label{source_tilde}
\end{align}

By definition of $h$ in \eqref{h_def}, we apply \eqref{h_l2_bdd_2} to have
\begin{align}
    & 
    \eqref{L2_energy}_4 = \Big| \int_0^t \int_{\mathbb{R}^3} \int_{\O}  h R  \dd x \dd v \dd s \Big| \notag\\
    & \lesssim T \Vert h\Vert_{L^\infty_{t}L^2_{x,v}}\Vert R\Vert_{L^\infty_{t}L^2_{x,v}}  \lesssim  o(1)\Vert R\Vert_{L^\infty_t L^2_{x,v}} + C_3T\e^{1/2}. \label{source_h}
\end{align}

For $\Gamma(R,R)$, we apply Lemma \ref{lemma:source_estimate} to have
\begin{align}
    & \eqref{L2_energy}_5 =  \Big|\int_0^t\int_{\mathbb{R}^3} \int_{\O} \e^{7/4}\Gamma(R,R) \e^{-1} (\mathbf{I}-\P)R  \dd x \dd v \dd s \Big| \notag\\
    & \lesssim \e^{7/2} \int_0^t \Vert \nu^{-1/2}\Gamma(R,R)\Vert_{L^2_{x,v}}^2 \dd s + o(1)\Vert \e^{-1} (\mathbf{I}-\P)R\Vert_{L^2_{t,x,\nu}}^2  \notag\\
    & \lesssim  T\e^{1/2}\Vert R\Vert_{L^\infty_{t}L^2_{x,v}}^2 \Vert \e^{3/2} wR\Vert_{L^\infty_{t,x,v}}^2 + \e^2 \Vert \e^{-1}(\mathbf{I}-\P)R\Vert_{ L^2_{t,x,v}}^2 \Vert \e^{3/2}wR\Vert_{L^\infty_{t,x,\nu}}^2 \notag\\
    &+ o(1)\Vert \e^{-1} (\mathbf{I}-\P)R\Vert_{L^2_{t,x,\nu}}^2  .   \label{source_gamma}
\end{align}

Plugging the estimate above into \eqref{L2_energy}, with $\e  \ll 1$ and $u_b\ll 1$, we conclude that 
\begin{align}
    & \Vert R(t)\Vert^2_{L^2_{x,v}}  + \Vert \e^{-1} (\mathbf{I}-\P)R\Vert^2_{L^2_{t,x,\nu}} + |\e^{-1/2}(I-P_\gamma)R|^2_{L^2_{t,\gamma_+}} \notag\\
    & \lesssim \Vert R_0\Vert_{L^2_{x,v}}^2  + C_3T\e^{1/2} + T|u_b|^2\Vert R\Vert^2_{L^\infty_t L^2_{x,v}} \notag\\
    &+  [\e^2\Vert \e^{-1}(\mathbf{I}-\P)R\Vert_{L^2_{t,x,\nu}}^2 + T\e^{1/2}\Vert R\Vert_{L^\infty_t L^2_{x,v}}^2] \Vert  \e^{3/2} wR\Vert_{L^\infty_{t,x,v}}^2  .    \label{energy_txv}
\end{align}

Using $\sqrt{T}|u_b|\ll 1$ and $\e \ll 1$, we conclude the lemma. 
\end{proof}

In view of Lemma \ref{lemma:l2}, it remains to estimate the $L^\infty_{t,x,v}$ norm in $\Vert R\Vert_{\infty,t}$.
\begin{lemma}\label{lemma:linfty_R}
Under the a priori assumption \eqref{apriori_assumption}, it holds that
\begin{align}
    & \Vert R\Vert_{\infty,t} \lesssim C_3\e^{2} + \Vert \e^{3/2}w R_0\Vert_{L^\infty_{x,v}}
+ \e\Vert R\Vert_{\infty,t}^2. \label{pointwise_norm_bdd}
\end{align}

\end{lemma}

\begin{proof}
We apply Lemma \ref{lemma:linfty} with $g,q$ given in \eqref{g_nonlinear}, \eqref{q_nonlinear}, so as to have
\begin{align*}
    & \Vert \e^{3/2}wR(t)\Vert_{L^\infty_{x,v}} \lesssim \Vert \e^{3/2}wR_0\Vert_{L^\infty_{x,v}} +   |\e^{3/2}wq|_{L^\infty_{t,x,v}} + \Vert R\Vert_{L^\infty_t L^2_{x,v}}\notag \\
    &+ \e^{5/2} \Vert \langle v\rangle^{-1} w [\e h+\tilde{\mathcal{L}}R + \e^{7/4}\Gamma(R,R)]\Vert_{L^\infty_{t,x,v}}   \notag.
\end{align*}
Applying Lemma \ref{lemma:q_estimate} to $q$, and Lemma \ref{lemma:source_estimate} to the source term $g$, we further have
\begin{align}
&\Vert \e^{3/2}wR(t)\Vert_{L^\infty_{x,v}} \lesssim \Vert \e^{3/2} wR_0\Vert_{L^\infty_{x,v}} + C_3\e^{2}  + o(1)\Vert \e^{3/2}wR\Vert_{L^\infty_{t,x,v}}   \notag\\
    &+   \Vert R\Vert_{L^\infty_t L^2_{x,v}} + \e^{5/2}\Vert \langle v\rangle^{-1}w(\e h+\tilde{\mathcal{L}}(R)+\e^{7/4}\Gamma(R,R)) \Vert_{L^\infty_{t,x,v}} \notag\\
    & \lesssim \Vert \e^{3/2} wR_0\Vert_{L^\infty_{x,v}} + C_3\e^{2} + o(1)\Vert \e^{3/2}wR\Vert_{L^\infty_{t,x,v}} \notag \\
    &+ \Vert R\Vert_{L^\infty_t L^2_{x,v}} + 
  \e^{5/4}  \Vert \e^{3/2} w R\Vert_{L^\infty_{t,x,v}}^2. \label{linfty_step_1}
\end{align}

We conclude the lemma.
\end{proof}

Now we are ready to prove Proposition \ref{prop:norm_bdd}.

\begin{proof}[\textbf{Proof of Proposition \ref{prop:norm_bdd}}]

By adding \eqref{l2_all_bdd} and  $\delta_2 \times \eqref{pointwise_norm_bdd}$ for $\delta_2 \ll 1$, we obtain
\begin{align*}
   \Vert R\Vert_{E,t} + \delta_2 \Vert R\Vert_{\infty,t} &  \lesssim  C_3T\e^{1/2}+ \Vert \e^{3/2}wR_0\Vert_{L^\infty_{x,v}} + \Vert R_0\Vert_{L^2_{x,v}} \\
    &+ \delta_2 \Vert R\Vert_{E,t} +  T\e^{1/2}\Vert R\Vert_{E,t}^2 + T\e^{1/2}\Vert R\Vert_{\infty,t}^2.
\end{align*}
This leads to the conclusion of Proposition \ref{prop:norm_bdd}.
\end{proof}

\subsection{Proof of Theorem \ref{thm:well-posedness}}\label{sec:proof_thm}
For any given $T$, we choose $\e$ to be small enough such that
\begin{align*}
    & T\e^{1/2} \ll o(1)[\Vert \e^{3/2}wR_0\Vert_{L^\infty_{x,v}} + \Vert R_0\Vert_{L^2_{x,v}}] \\
    & T\e^{1/2} [\Vert \e^{3/2}wR_0\Vert_{L^\infty_{x,v}} + \Vert R_0\Vert_{L^2_{x,v}}] \ll 1 .\end{align*}
In the a priori estimate in Proposition \ref{prop:norm_bdd}, $C_3=C_3(\delta,T)$, this implies that there exists $C_4=C_4(\delta,T)$ such that
\begin{align*}
    &     \Vert R\Vert_{E,t} + \Vert R\Vert_{\infty,t}  \\
    &\leq C_4 [T\e^{1/2}+ \Vert \e^{3/2}wR_0\Vert_{L^\infty_{x,v}} + \Vert R_0\Vert_{L^2_{x,v}}  +  T\e^{1/2}\Vert R\Vert_{E,t}^2 + T\e^{1/2}\Vert R\Vert_{\infty,t}^2].
\end{align*}

Assuming the a priori assumption $\Vert R\Vert_{E,t} + \Vert R\Vert_{\infty,t}<2C_4 [\Vert \e^{3/2}wR_0\Vert_{L^\infty_{x,v}} + \Vert R_0\Vert_{L^2_{x,v}}]$, then we arrive at the a priori estimate
\begin{align*}
    &  \Vert R\Vert_{E,t}+\Vert R\Vert_{\infty,t} < C_4 T\e^{1/2} + C_4[\Vert \e^{3/2}wR_0\Vert_{L^\infty_{x,v}} + \Vert R_0\Vert_{L^2_{x,v}}] \\
    &+ 2C_4T\e^{1/2}[\Vert \e^{3/2}wR_0\Vert_{L^\infty_{x,v}} + \Vert R_0\Vert_{L^2_{x,v}}][\Vert R\Vert_{E,t}+\Vert R\Vert_{\infty,t}] \\ 
    & <(C_4+o(1))[\Vert \e^{3/2}wR_0\Vert_{L^\infty_{x,v}} + \Vert R_0\Vert_{L^2_{x,v}}] + o(1)[\Vert R\Vert_{E,t}+\Vert R\Vert_{\infty,t}]  .
\end{align*}
This leads to the same conclusion as the a priori assumption
\begin{align*}
    &  \Vert R\Vert_{E,t}+\Vert R\Vert_{\infty,t} < 2C_4 [\Vert \e^{3/2}wR_0\Vert_{L^\infty_{x,v}} + \Vert R_0\Vert_{L^2_{x,v}}].
\end{align*}
This proves \eqref{R_est} with $C_1 = 2C_4$.

Then it is standard to apply the sequential argument or fixed point theorem to construct a unique solution to \eqref{remainder_eqn} that satisfies \eqref{R_est}. For detailed construction of the well-posedness and positivity, we refer readers to \cite{EGKM2}.

Since the expansion coefficients $f_1,f_i$ in \eqref{f1}, \eqref{f_i} are completely determined by $u$ and the corresponding initial condition $f_{i,0}$, which has a unique solution as the heat equation and linear NSF system in \eqref{u_p} and \eqref{u_theta_i_eqn}, we conclude the existence and uniqueness for the original Boltzmann equation $F$ in \eqref{eqn_F}. The positivity of the solution can also be justified by employing the classical positive-preserving iteration scheme in \cite{EGKM2}.

Finally, the convergence from the kinetic equation to the fluid equation in $\e$ \eqref{convergence} follows from 
\begin{align*}
   & \Big\Vert \frac{F-\mu}{\e\sqrt{\mu}} - f_1\Big\Vert_{L^\infty_t L^2_{x,v}} \leq \e \Vert f_2\Vert_{L^\infty_t L^2_{x,v}} + \e^{3/2+\frac{1}{4}}\Vert R\Vert_{L^\infty_t L^2_{x,v}} \lesssim \e, \\
    &  \Big\Vert \frac{F-\mu}{\e\sqrt{\mu}} - f_1\Big\Vert_{L^\infty_{t,x,v}} \leq \e \Vert f_2\Vert_{L^\infty_{t,x,v}} + \e^{3/2+\frac{1}{4}}\Vert R\Vert_{L^\infty_{t,x,v}} \lesssim \e^{1/4}.
\end{align*}
for any $0\leq t\leq T$. Here we have applied \eqref{R_est} and Lemma \ref{lemma:inter_solution}.

The proof of Theorem \ref{thm:well-posedness} is complete. \qed

\appendix

\section{Hilbert expansion}
The source term in \eqref{u_theta_i_eqn} are given by
\begin{align}
  S_u^{i}  &     = \langle     v\cdot \nabla_x \mathcal{L}^{-1} [\p_t (\mathbf{I}-\mathbf{P})f_{i-1} + (\mathbf{I}-\mathbf{P})(v\cdot \nabla_x (\mathbf{I}-\mathbf{P})f_i)],v\sqrt{\mu}\rangle \notag\\
    & - \langle v\cdot \nabla_x \mathcal{L}^{-1}(2\Gamma(f_1,(\mathbf{I}-\mathbf{P})f_i) + \sum_{j+k=i+1, \ 1<j,k}\Gamma(f_j,f_k)),v\sqrt{\mu}\rangle \notag\\
    & - (\p_t + u\cdot \nabla_x - \kappa \Delta)(I-P_0)u_i - (I-P_0)u_i\cdot \nabla_x u  \notag\\
    &  - (\nabla_x \cdot (I-P_0)u_i)u + \frac{\kappa}{3}\nabla_x (\nabla_x \cdot (I-P_0)u_i)  , \label{S_u}\\
   S_\theta^{i} & = \langle     v\cdot \nabla_x \mathcal{L}^{-1} [\p_t (\mathbf{I}-\mathbf{P})f_{i-1} + (\mathbf{I}-\mathbf{P})(v\cdot \nabla_x (\mathbf{I}-\mathbf{P})f_i)],\frac{|v|^2}{5}\sqrt{\mu}\rangle \notag\\
    & - \langle v\cdot \nabla_x \mathcal{L}^{-1}(2\Gamma(f_1,(\mathbf{I}-\mathbf{P})f_i) + \sum_{j+k=i+1, \ 1<j,k}\Gamma(f_j,f_k)),\frac{|v|^2}{5}\sqrt{\mu}\rangle, \notag\\
    & + \frac{2}{5}\p_t (\rho_i+\theta_i)  \label{S_theta}.
\end{align}

\section{Stationary profile of the Rayleigh problem}\label{sec:stationary}

In case of a finite channel domain $\Omega=\mathbb{T}^2\times (0,1)$ with tangent shear moving on the boundaries, the long-time behavior for the initial boundary value problem on the Boltzmann equation is determined by the corresponding steady solution; see recent study \cite{duan20243d}. However, for the half-space Rayleigh problem under consideration in the current work, it seems impossible to expect the same property. In fact, for conciseness, we may consider the one-dimensional steady Boltzmann equation in the half-line without any scaling:
\begin{align}\label{bvp.eq}
v_2 \p_{y} F = Q(F,F),
\end{align}
where $F=F(y,v)\geq 0$ for $y\in \mathbb{R}^+$, $v = (v_1,v_2,v_3)\in \mathbb{R}^3$. On the boundary $y=0$, we impose the diffuse boundary condition
\begin{align}\label{bvp.bc}
    &  F(0,v)|_{v_2>0} = M(v_1 - \alpha, v_2,v_3) \int_{u_2<0} F(0,u) |u_2| \,\dd u, 
\end{align}
where $M(v)$ is the wall Maxwellian given by
\begin{align*}
    &   M(v_1,v_2,v_3) = \frac{1}{2\pi} \exp\Big(-\frac{v_1^2+v_2^2+v_3^2}{2}\Big),
\end{align*}
and $\alpha>0$ is the velocity of the wall in the $v_1$ direction.

We consider whether or not there exists a stationary profile of the problem where the gas is in the rest equilibrium state in the far field:
\begin{align}\label{ap.inf}
    F(y,v)\to \mu(v) \text{ as } y\to + \infty.
\end{align}
Note that the boundary layer problem \eqref{bvp.eq} and \eqref{bvp.bc} supplemented with the extra far-field condition \eqref{ap.inf} could be overdetermined. 

Indeed, we first find a function $U$ that connects the tangential flow velocity $U(y)$ from $U(0)=1$ to $\lim_{y\to\infty} U(y)=0$. For example, this function $U(y)$ can be constructed in the following way:
\begin{align}\label{def.U}
    & U(y) := \frac{2}{\sqrt{\pi}}\int_{y}^\infty  e^{-r^2} \dd r,\ y\geq 0; \ U(0) = 1, \ \lim_{y\to +\infty} U(y) = 0.
\end{align}

Equivalently we look for steady solutions of the following shear profile:
\begin{align*}
    & F_{st}(y,v_1 - \alpha U(y), v_2,v_3).
\end{align*}
Plugging this into \eqref{bvp.eq} and \eqref{bvp.bc}, we obtain the boundary-value problem on $F_{st}$ as
\begin{align*}
\begin{cases}
        &\displaystyle  v_2 \p_y F_{st}  - \alpha v_2 U'(y) \p_{v_1} F_{st} = Q(F_{st},F_{st}), \ y\in (0,\infty), \ v\in \mathbb{R}^3, \\
    &\displaystyle F_{st}(0,v)|_{v_2>0} = \sqrt{2\pi}\mu \int_{u_2<0} F_{st}(0,u) |u_2| \,\dd u.
\end{cases}
\end{align*}

Assuming that $0<\alpha \ll 1$, we then expand $F_{st}$ in $\alpha$:
\begin{align*}
    & F_{st} = \mu + \sqrt{\mu} (\alpha G_1 + |\alpha|^2 G_R),
\end{align*}
with
\begin{align*}
    & \int_{\mathbb{R}^+}\int_{\mathbb{R}^3} \sqrt{\mu}G_1 \,\dd v \dd y = \int_{\mathbb{R}^+}\int_{\mathbb{R}^3}\sqrt{\mu} G_R \,\dd v \dd y = 0.
\end{align*}

Comparing the order of $\alpha$, we obtain the equation for $G_1$,
\begin{align*}
    & v_2 \p_y G_1 + \mathcal{L} G_1 =  - U'(y) v_1 v_2 \sqrt{\mu},
\end{align*}
with boundary condition
\begin{align*}
       & G_1(0,v)|_{v_2>0} = \sqrt{2\pi \mu(v)} \int_{u_2<0} \sqrt{\mu}G_1(0,u)|u_2| \,\dd u.
\end{align*}
From the oddness in $v_1$ in the equation of $G_1$, this boundary condition reduces to 
\[G_1(0,v)|_{v_2>0} = 0.\]

Note that $G_1$ corresponds to the Milne's problem \cite{bardos1986milne} with a source term. In fact, for this problem we can find an explicit solution:
\begin{align}
    & G_1 = (1-U(y)) v_1 \sqrt{\mu}. \notag 
\end{align}
\hide

\Red{Hi HX: May I ask whether one can modify the boundary condition of $G_1$ as 
\begin{align}\label{G1.mo}
G_1(0,v)|_{v_2>0} = \sqrt{2\pi \mu(v)} \int_{u_2<0} \sqrt{\mu}G_1(0,u)|u_2| \dd u-v_1\sqrt{\mu},
\end{align}
such that $G_1=-U(y)v_1\sqrt{\mu}$? And then, we find a suitable boundary condition to solve the remainder? RJ}
\Blue{In this setting the expansion will become
\begin{align*}
    &  F_{st} = \mu + \sqrt{\mu}(\alpha (G_1 - v_1\sqrt{\mu}) + \alpha^2(G_R + \frac{1}{\alpha}v_1\sqrt{\mu})).
\end{align*}
Then I think this will cause the new remainder $G_R+\frac{1}{\alpha}v_1\sqrt{\mu}$ to have a scale $\frac{1}{\alpha}$ then, I am not sure if this will go back to the same situation.}\Red{We may include a linear boundary layer around the boundary, to take away such boundary inhomogeneous boundary term $-v_1\sqrt{\mu}$ in \eqref{G1.mo} (PS: the linear boundary layer $\cdot(y/\epsilon)$ does not change the equation and the far-field). We may consider such possibility. }

\unhide
By \eqref{def.U} it holds 
$$
\lim_{y\to \infty} G_1(y,v) = v_1 \sqrt{\mu}.
$$ 
Hence, in the leading order, the solution $F_{st}$, if it exists, has to behave as
\begin{align*}
    & F_{st} \sim \mu + \alpha(1-U(y)) v_1 \mu, \\
    & \lim_{y\to \infty} F_{st} \sim \mu(v) + \alpha v_1 \mu(v) \sim \mu(v_1 - \alpha, v_2, v_3).
\end{align*}
This implies that the far field condition \eqref{ap.inf} is not appropriate to ensure the solvability of the boundary-value problem \eqref{bvp.eq} and \eqref{bvp.bc}. Therefore, the gas should move in the same flow velocity as the boundary; in such case, there only exists a trivial stationary solution $\mu(v_1-\alpha,v_2,v_3)$.

\hide

{\color{red}
Difficulty:
\begin{enumerate}
    \item Boundary condition becomes $R = P_\gamma R + O(\e)P_\gamma R + r$, in the test function method, there will be an extra term $O(\e)|P_\gamma R|_{2,+}$.

    \item Need to use Ukai trace theorem to control the extra term, this leads to a control of $O(\e)\int_0^t \Vert R\Vert^2_2$. 

    \item Such control is impossible for $t\to \infty$. With an extra $\e$, we can let the upper bound of $t$ depend on $\e$. 
\end{enumerate}

}

\newpage


For the stationary problem
\begin{align*}
    & v\cdot \nabla_x F = \frac{Q(F,F)}{\e},
\end{align*}
the asymptotic expansion
\begin{align*}
    & F = \mu + \sqrt{\mu}(\e f_1 + \e^2 f_2 + \cdots + \e^k f_k) + \sqrt{\mu} \delta R
\end{align*}
leads 
\begin{align*}
    &    \text{ Order of $\e^0$: } \mathcal{L}f_1 = 0, \label{order_0}\\
    &    \text{ Order of $\e^1$: } v\cdot \nabla_x f_1 + \mathcal{L} f_2 = \Gamma(f_1,f_1),  \\
    &    \text{ Order of $\e^2$: }  + v\cdot \nabla_x f_2 + \mathcal{L} f_3 = \Gamma(f_1,f_2) + \Gamma(f_2,f_1) \\
    & \vdots \notag\\
    & \text{ Order of $\e^{k-1}$: }  v\cdot \nabla_x f_{k-1} + \mathcal{L} f_{k} = \sum_{i+j=k, i,j\geq 1} \big( \Gamma(f_i,f_j) + \Gamma(f_j,f_i) \big).  
\end{align*}

From a standard computation, \eqref{order_1} leads to the incompressibility condition 
\begin{align*}
    & \nabla_x \cdot u_p = 0,
\end{align*}
and \eqref{order_i} leads to the incompressible Navier-Stokes equation:
\begin{align*}
    &  (u_p\cdot \nabla_x) u_p + \nabla_x p_1 = \kappa \Delta u_p,  \\   
    & \nabla_x p_1 = \nabla_x (\rho_2 + \theta_2 - \frac{1}{3}|u_p|^2). \notag
\end{align*}

In general, for order of $\e^{l},\e^{l+1}, 2\leq l\leq k-1$, we denote
\begin{align*}
    \mathbf{P}(f_l) = \big(\rho_l + u_l\cdot v + \theta_l \frac{|v|^2-3}{2} \big) \sqrt{\mu}.
\end{align*}

The conservation law leads to the incompressibility condition, Boussinesq relation and the linear Navier-Stoke-Fourier system:
\begin{align*}
    &  \nabla_x \cdot u_l = 0, \\
    &   (u_p \cdot \nabla_x ) u_l + (u_l\cdot \nabla_x) u_p + \nabla_x p_l - \kappa \Delta u_l = R_l^u,\\
    & R_l^u = \Big\langle    v\cdot \nabla_x \mathcal{L}^{-1} \big( v\cdot \nabla_x (\mathbf{I}-\mathbf{P})f_l \big)  , v\sqrt{\mu}  \Big\rangle \\
    & + \Big\langle    v\cdot \nabla_x \mathcal{L}^{-1} \big( \Gamma(f_1,(\mathbf{I}-\mathbf{P})f_l) + \Gamma((\mathbf{I}-\mathbf{P})f_l,f_1) \big)  , v\sqrt{\mu}  \Big\rangle \\
    & + \Big\langle    v\cdot \nabla_x \mathcal{L}^{-1} \big( \sum_{i+j=l+1,i,j\geq 2}\big[\Gamma(f_i,f_j) + \Gamma(f_j,f_i) \big] \big)  , v\sqrt{\mu}  \Big\rangle 
\end{align*}
\begin{align*}
    &   \nabla_x p_l = \nabla_x (\rho_{l+1}+\theta_{l+1} - \frac{2-\delta_{l1}}{3}(u_l\cdot u_p)), \notag\\
    &   \rho_{l+1}+\theta_{l+1} = p_l + \frac{2-\delta_{l1}}{3} (u_l\cdot u_p)+C_{l+1}, \text{ for some constant }C_{l+1}. 
\end{align*}

Since $\theta_1=0$, the heat flow for higher order is given by
\begin{align*}
    &  u_p \cdot \nabla_x \theta_l - \lambda \Delta \theta_l  =   R_l^\theta \\
    & R_l^\theta = \Big\langle    v\cdot \nabla_x \mathcal{L}^{-1} \big( v\cdot \nabla_x (\mathbf{I}-\mathbf{P})f_l \big)  , \frac{|v|^2}{5}\sqrt{\mu}  \Big\rangle \\
    & + \Big\langle    v\cdot \nabla_x \mathcal{L}^{-1} \big( \Gamma(f_1,(\mathbf{I}-\mathbf{P})f_l) + \Gamma((\mathbf{I}-\mathbf{P})f_l,f_1) \big)  , \frac{|v|^2\sqrt{\mu}}{5}  \Big\rangle \\
    & + \Big\langle    v\cdot \nabla_x \mathcal{L}^{-1} \big( \sum_{i+j=l+1,i,j\geq 2}\big[\Gamma(f_i,f_j) + \Gamma(f_j,f_i) \big] \big)  , \frac{|v|^2\sqrt{\mu}}{5}  \Big\rangle .
\end{align*}

We stop the expansion at $k=3$: $F=\mu+ \sqrt{\mu} (\e f_1 + \e^2 f_2 + \e^2 \mathcal{F}_2 + \e^3 f_3) + \e^2 \sqrt{\mu} R$. The equation of $R$ is given by
\begin{align*}
    v\cdot \nabla_x R + \frac{1}{\e} \mathcal{L}(R) = \e \Gamma(R,R) + S,
\end{align*}
with
\begin{align*}
    &   
\end{align*}
\unhide

\medskip
\noindent {\bf Acknowledgment:}\,
The research of Renjun Duan was partially supported by the General Research Fund (Project No.~14301822) from RGC of Hong Kong and also by the grant from the National Natural Science Foundation of China (Project No.~12425109).

\medskip
\noindent{\bf Data availability:} The manuscript contains no associated data.

\medskip
\noindent{\bf Conflict of Interest:} The authors declare that they have no conflict of interest.

\bibliographystyle{siam}


\end{document}